\documentclass[12pt]{amsart}       
\usepackage{txfonts}
\usepackage{amssymb}
\usepackage{eucal}
\usepackage{graphicx}
\usepackage{amsmath}
\usepackage{amscd}
\usepackage[all]{xy}           

\usepackage{amsfonts,latexsym}
\usepackage{xspace}
\usepackage{epsfig}
\usepackage{float}
\usepackage{color}
\usepackage{fancybox}
\usepackage{colordvi}
\usepackage{multicol}
\usepackage{colordvi}
\usepackage[active]{srcltx} 
\usepackage[colorlinks,final,backref=page,hyperindex,hypertex]{hyperref}

\newcommand{\nc}{\newcommand}
\newcommand{\delete}[1]{}

\nc{\mlabel}[1]{\label{#1}}  
\nc{\mcite}[1]{\cite{#1}}  
\nc{\mref}[1]{\ref{#1}}  
\nc{\mbibitem}[1]{\bibitem{#1}} 

\delete{
\nc{\mlabel}[1]{\label{#1}  
{\hfill \hspace{1cm}{\small\tt{{\ }\hfill(#1)}}}}
\nc{\mcite}[1]{\cite{#1}{\small{\tt{{\ }(#1)}}}}  
\nc{\mref}[1]{\ref{#1}{{\tt{{\ }(#1)}}}}  
\nc{\mbibitem}[1]{\bibitem[\bf #1]{#1}} 
}

\newtheorem{theorem}{Theorem}[section]
\newtheorem{prop}[theorem]{Proposition}
\newtheorem{lemma}[theorem]{Lemma}
\newtheorem{coro}[theorem]{Corollary}
\theoremstyle{definition}
\newtheorem{defn}[theorem]{Definition}
\newtheorem{prop-def}{Proposition-Definition}[section]

\newtheorem{remark}[theorem]{Remark}

\newtheorem{tempex}[theorem]{Example}
\newtheorem{tempexs}[theorem]{Examples}
\newtheorem{temprmk}[theorem]{Remark}
\newtheorem{tempexer}{Exercise}[section]

\nc{\vsa}{\vspace{-.1cm}} \nc{\vsb}{\vspace{-.2cm}}
\nc{\vsc}{\vspace{-.3cm}} \nc{\vsd}{\vspace{-.4cm}}
\nc{\vse}{\vspace{-.5cm}}

\nc{\NS}{U_{NS}}
\nc{\FN}{F_{\mathrm Nij}}
\nc{\dfgen}{V} \nc{\dfrel}{R}
\nc{\dfgenb}{\vec{v}} \nc{\dfrelb}{\vec{r}}
\nc{\dfgene}{v} \nc{\dfrele}{r}
\nc{\dfop}{\odot}
\nc{\dfoa}{\dfop^{(1)}} \nc{\dfob}{\dfop^{(2)}}
\nc{\dfoc}{\dfop^{(3)}} \nc{\dfod}{\dfop^{(4)}}
\nc{\mapm}[1]{\lfloor\!|{#1}|\!\rfloor}
\nc{\cmapm}[1]{\frakC(#1)}
\nc{\red}{\mathrm{Red}}
\nc{\cm}{C}
\nc{\supp}{\mathrm{Supp}}
\nc{\lex}{\mathrm{lex}}

\nc{\disp}[1]{\displaystyle{#1}}
\nc{\bin}[2]{ (_{\stackrel{\scs{#1}}{\scs{#2}}})}  
\nc{\binc}[2]{ \left (\!\! \begin{array}{c} \scs{#1}\\
    \scs{#2} \end{array}\!\! \right )}  
\nc{\bincc}[2]{  \left ( {\scs{#1} \atop
    \vspace{-.5cm}\scs{#2}} \right )}  
\nc{\sarray}[2]{\begin{array}{c}#1 \vspace{.1cm}\\ \hline
    \vspace{-.35cm} \\ #2 \end{array}}
\nc{\bs}{\bar{S}} \nc{\ep}{\epsilon}
\nc{\dbigcup}{\stackrel{\bullet}{\bigcup}}
\nc{\la}{\longrightarrow} \nc{\cprod}{\ast} \nc{\rar}{\rightarrow}
\nc{\dar}{\downarrow} \nc{\labeq}[1]{\stackrel{#1}{=}}
\nc{\dap}[1]{\downarrow \rlap{$\scriptstyle{#1}$}}
\nc{\uap}[1]{\uparrow \rlap{$\scriptstyle{#1}$}}
\nc{\defeq}{\stackrel{\rm def}{=}} \nc{\dis}[1]{\displaystyle{#1}}
\nc{\dotcup}{\ \displaystyle{\bigcup^\bullet}\ }
\nc{\sdotcup}{\tiny{ \displaystyle{\bigcup^\bullet}\ }}
\nc{\fe}{\'{e}}
\nc{\hcm}{\ \hat{,}\ } \nc{\hcirc}{\hat{\circ}}
\nc{\hts}{\hat{\shpr}} \nc{\lts}{\stackrel{\leftarrow}{\shpr}}
\nc{\denshpr}{\den{\shpr}}
\nc{\rts}{\stackrel{\rightarrow}{\shpr}} \nc{\lleft}{[}
\nc{\lright}{]} \nc{\uni}[1]{\tilde{#1}} \nc{\free}[1]{\bar{#1}}
\nc{\freea}[1]{\tilde{#1}} \nc{\freev}[1]{\hat{#1}}
\nc{\dt}[1]{\hat{#1}}
\nc{\wor}[1]{\check{#1}}
\nc{\intg}[1]{F_C(#1)}
\nc{\den}[1]{\check{#1}} \nc{\lrpa}{\wr} \nc{\mprod}{\pm}
\nc{\dprod}{\ast_P} \nc{\curlyl}{\left \{ \begin{array}{c} {} \\
{} \end{array}
    \right .  \!\!\!\!\!\!\!}
\nc{\curlyr}{ \!\!\!\!\!\!\!
    \left . \begin{array}{c} {} \\ {} \end{array}
    \right \} }
\nc{\longmid}{\left | \begin{array}{c} {} \\ {} \end{array}
    \right . \!\!\!\!\!\!\!}
\nc{\lin}{\call} \nc{\ot}{\otimes}
\nc{\ora}[1]{\stackrel{#1}{\rar}}
\nc{\ola}[1]{\stackrel{#1}{\la}}
\nc{\scs}[1]{\scriptstyle{#1}} \nc{\mrm}[1]{{\rm #1}}
\nc{\margin}[1]{\marginpar{\rm #1}}   
\nc{\dirlim}{\displaystyle{\lim_{\longrightarrow}}\,}
\nc{\invlim}{\displaystyle{\lim_{\longleftarrow}}\,}
\nc{\mvp}{\vspace{0.5cm}}
\nc{\mult}{m}       
\nc{\svp}{\vspace{2cm}} \nc{\vp}{\vspace{8cm}}
\nc{\proofbegin}{\noindent{\bf Proof: }}
\nc{\proofend}{$\blacksquare$ \vspace{0.5cm}}
\nc{\sha}{{\mbox{\cyr X}}}  
\nc{\ncsha}{{\mbox{\cyr X}^{\mathrm NC}}}
\newfont{\scyr}{wncyr10 scaled 550}
\nc{\ssha}{\mbox{\bf \scyr X}}
\nc{\ncshao}{{\mbox{\cyr
X}^{\mathrm NC,\,0}}}
\nc{\shpr}{\diamond}    
\nc{\shprc}{\shpr_c}
\nc{\shpro}{\diamond^0}    
\nc{\shpru}{\check{\diamond}} \nc{\spr}{\cdot}
\nc{\catpr}{\diamond_l} \nc{\rcatpr}{\diamond_r}
\nc{\lapr}{\diamond_a} \nc{\lepr}{\diamond_e} \nc{\sprod}{\bullet}
\nc{\un}{u}                 
\nc{\vep}{\varepsilon} \nc{\labs}{\mid\!} \nc{\rabs}{\!\mid}
\nc{\hsha}{\widehat{\sha}} \nc{\lsha}{\stackrel{\leftarrow}{\sha}}
\nc{\rsha}{\stackrel{\rightarrow}{\sha}} \nc{\lc}{\lfloor}
\nc{\rc}{\rfloor} \nc{\sqmon}[1]{\langle #1\rangle}
\nc{\altx}{\Lambda} \nc{\vecT}{\vec{T}} \nc{\piword}{{\mathfrak P}}
\nc{\lbar}[1]{\overline{#1}}
\nc{\dep}{\mathrm{dep}}

\nc{\mmbox}[1]{\mbox{\ #1\ }}
\nc{\ayb}{\mrm{AYB}} \nc{\mayb}{\mrm{mAYB}} \nc{\cyb}{\mrm{cyb}}
\nc{\ann}{\mrm{ann}} \nc{\Aut}{\mrm{Aut}} \nc{\cabqr}{\mrm{CABQR
}} \nc{\can}{\mrm{can}} \nc{\colim}{\mrm{colim}}
\nc{\Cont}{\mrm{Cont}} \nc{\rchar}{\mrm{char}}
\nc{\cok}{\mrm{coker}} \nc{\dtf}{{R-{\rm tf}}} \nc{\dtor}{{R-{\rm
tor}}}

\nc{\Div}{{\mrm Div}} \nc{\End}{\mrm{End}} \nc{\Ext}{\mrm{Ext}}
\nc{\FG}{\mrm{FG}} \nc{\Fil}{\mrm{Fil}} \nc{\Frob}{\mrm{Frob}}
\nc{\Gal}{\mrm{Gal}} \nc{\GL}{\mrm{GL}} \nc{\Hom}{\mrm{Hom}}
\nc{\hsr}{\mrm{H}} \nc{\hpol}{\mrm{HP}} \nc{\id}{\mrm{id}} \nc{\Id}{\mathrm{Id}}
\nc{\im}{\mrm{im}} \nc{\incl}{\mrm{incl}} \nc{\Loday}{\mrm{ABQR}\
} \nc{\length}{\mrm{length}} \nc{\LR}{\mrm{LR}} \nc{\mchar}{\rm
char} \nc{\pmchar}{\partial\mchar} \nc{\map}{\mrm{Map}}
\nc{\MS}{\mrm{MS}} \nc{\OS}{\mrm{OS}} \nc{\NC}{\mrm{NC}}
\nc{\rba}{\rm{Rota-Baxter algebra}\xspace}
\nc{\rbas}{\rm{Rota-Baxter algebras}\xspace}
\nc{\rbw}{\rm{RBW}\xspace}
\nc{\rbws}{\rm{RBWs}\xspace}
\nc{\rbadj}{\rm{RB}\xspace}
\nc{\mpart}{\mrm{part}} \nc{\ql}{{\QQ_\ell}} \nc{\qp}{{\QQ_p}}
\nc{\rank}{\mrm{rank}} \nc{\rcot}{\mrm{cot}} \nc{\rdef}{\mrm{def}}
\nc{\rdiv}{{\rm div}} \nc{\rtf}{{\rm tf}} \nc{\rtor}{{\rm tor}}
\nc{\res}{\mrm{res}} \nc{\SL}{\mrm{SL}} \nc{\Spec}{\mrm{Spec}}
\nc{\tor}{\mrm{tor}} \nc{\Tr}{\mrm{Tr}}
\nc{\mtr}{\mrm{tr}}

\nc{\ab}{\mathbf{Ab}} \nc{\Alg}{\mathbf{Alg}}
\nc{\Bax}{\mathbf{CRB}} \nc{\Algo}{\mathbf{Alg}^0}
\nc{\cRB}{\mathbf{CRB}} \nc{\cRBo}{\mathbf{CRB}^0}
\nc{\RBo}{\mathbf{RB}^0} \nc{\BRB}{\mathbf{RB}}
\nc{\Dend}{\mathbf{DD}} \nc{\bfk}{{\bf k}} \nc{\bfone}{{\bf 1}}
\nc{\base}[1]{{a_{#1}}} \nc{\Cat}{\mathbf{Cat}}
 \nc{\DN}{\mathbf{DN}}
\nc{\NA}{\mathbf{NA}}
\nc{\SDN}{\mathbf{SDN}}
\nc{\Diff}{\mathbf{Diff}} \nc{\gap}{\marginpar{\bf
Incomplete}\noindent{\bf Incomplete!!}
    \svp}
\nc{\FMod}{\mathbf{FMod}} \nc{\Int}{\mathbf{Int}}
\nc{\Mon}{\mathbf{Mon}}
\nc{\RB}{\mathbf{RB}} \nc{\remarks}{\noindent{\bf Remarks: }}
\nc{\Rep}{\mathbf{Rep}} \nc{\Rings}{\mathbf{Rings}}
\nc{\Sets}{\mathbf{Sets}} \nc{\bfx}{\mathbf{x}}
\nc{\BA}{{\Bbb A}} \nc{\CC}{{\Bbb C}} \nc{\DD}{{\Bbb D}}
\nc{\EE}{{\Bbb E}} \nc{\FF}{{\Bbb F}} \nc{\GG}{{\Bbb G}}
\nc{\HH}{{\Bbb H}} \nc{\LL}{{\Bbb L}} \nc{\NN}{{\Bbb N}}
\nc{\QQ}{{\Bbb Q}} \nc{\RR}{{\Bbb R}} \nc{\TT}{{\Bbb T}}
\nc{\VV}{{\Bbb V}} \nc{\ZZ}{{\Bbb Z}}

\nc{\cala}{{\mathcal A}} \nc{\calb}{{\mathcal B}}
\nc{\calc}{{\mathcal C}}
\nc{\cald}{{\mathcal D}} \nc{\cale}{{\mathcal E}}
\nc{\calf}{{\mathcal F}} \nc{\calg}{{\mathcal G}}
\nc{\calh}{{\mathcal H}} \nc{\cali}{{\mathcal I}}
\nc{\calj}{{\mathcal J}} \nc{\call}{{\mathcal L}}
\nc{\calm}{{\mathcal M}} \nc{\caln}{{\mathcal N}}
\nc{\calo}{{\mathcal O}} \nc{\calp}{{\mathcal P}}
\nc{\calr}{{\mathcal R}} \nc{\cals}{{\mathcal S}} \nc{\calt}{{\mathcal T}}
\nc{\calw}{{\mathcal W}} \nc{\calx}{{\mathcal X}}
\nc{\CA}{\mathcal{A}}

\nc{\frakA}{{\mathfrak A}}
\nc{\fraka}{{\mathfrak a}}
\nc{\frakB}{{\mathfrak B}}
\nc{\frakb}{{\mathfrak b}}
\nc{\frakC}{{\mathfrak C}}
\nc{\frakd}{{\mathfrak d}}
\nc{\frakF}{{\mathfrak F}}
\nc{\frakg}{{\mathfrak g}}
\nc{\frakm}{{\mathfrak m}}
\nc{\frakM}{{\mathfrak M}}
\nc{\frakMo}{{\mathfrak M}^0}
\nc{\frakP}{{\mathfrak P}}
\nc{\frakp}{{\mathfrak p}}
\nc{\frakS}{{\mathfrak S}}
\nc{\frakSo}{{\mathfrak S}^0}
\nc{\fraks}{{\mathfrak s}}
\nc{\os}{\overline{\fraks}}
\nc{\frakT}{{\mathfrak T}}
\nc{\frakTo}{{\mathfrak T}^0}
\nc{\oT}{\overline{T}}
\nc{\frakX}{{\mathfrak X}}
\nc{\frakXo}{{\mathfrak X}^0}
\nc{\frakx}{{\mathbf x}}
\nc{\frakTx}{\frakT}      
\nc{\frakTa}{\frakT^a}        
\nc{\frakTxo}{\frakTx^0}   
\nc{\caltao}{\calt^{a,0}}   
\nc{\ox}{\overline{\frakx}} \nc{\fraky}{{\mathfrak y}}
\nc{\frakz}{{\mathfrak z}} \nc{\oX}{\overline{X}} \font\cyr=wncyr10

\nc{\tred}[1]{\textcolor{red}{#1}}
\nc{\tblue}[1]{\textcolor{blue}{#1}} \nc{\li}[1]{\tred{Li:#1 }}
\nc{\xing}[1]{\tblue{Xing:#1 }}

\begin{document}

\title[Free commutative integro-differential algebras and Gr\"obner-Shirshov bases]{Construction of free commutative integro-differential algebras by the method of Gr\"{o}bner-Shirshov bases}

\author{Xing Gao}
\address{Department of Mathematics,
    Lanzhou University,
    Lanzhou, Gansu 730000, China}
\email{gaoxing@lzu.edu.cn}
\author{Li Guo}
\address{
Department of Mathematics and Computer Science,
Rutgers University,
Newark, NJ 07102, USA}
\email{liguo@rutgers.edu}
\author{Shanghua Zheng}
\address{Department of Mathematics,
    Lanzhou University,
    Lanzhou, Gansu 730000, China}
\email{zheng2712801@163.com}

\date{\today}

\begin{abstract}
In this paper, we construct a canonical linear basis for free commutative integro-differential algebras by applying the method of Gr\"obner-Shirshov bases. We establish the Composition-Diamond Lemma for free commutative differential Rota-Baxter algebras of order $n$. We also obtain a weakly monomial order on these algebras, allowing us to obtain Gr\"{o}bner-Shirshov bases for free commutative integro-differential algebras on a set.
We finally generalize the concept of functional derivations to free differential algebras with arbitrary weight and generating sets from which to construct a canonical linear basis for free commutative integro-differential algebras.
\end{abstract}

\subjclass[2010]{
16S15, 
13P10, 
16W99, 
17A50, 
12H05, 
47G20 
.}

\keywords{Differential algebra, Rota-Baxter algebra, integro-differential algebra, Gr\"obner-Shirshov basis, free algebra, shuffle product, mixable shuffle product.}

\maketitle

\tableofcontents

\setcounter{section}{0}

\section{Introduction}

\subsection{Integro-differential algebras}
The algebraic study in analysis has a long history. The first monograph~\mcite{Ri1} of Ritt on algebraic study of differential equations appeared almost one hundred years ago. The concept of a {\bf differential algebra} was abstracted from the Leibniz formula
\begin{equation}
 d(uv)=d(u)v+ud(v)
\mlabel{eq:diff}
\end{equation}
in calculus. After the fundamental works of Ritt~\mcite{Ri2} and Kolchin~\mcite{Ko}, the theory of differential algebra has been expanded to a vast area of pure and applied mathematical study~\mcite{CGKS,SP}.
The algebraic study of the integral analysis began with the concept of a Baxter algebra~\mcite{Ba}, later called a {\bf Rota-Baxter algebra}. Here the basis of abstraction is the integration by parts formula,
\begin{equation}
P(u)P(v)=P(uP(v))+P(P(u)v)+\lambda P(uv),
\mlabel{eq:rb0}
\end{equation}
rewritten in a form that only involves the integral operator $P$, defined by $P(u)(x):=\int_a^x u(t)\,dt$. The extra term parameterized by a constant $\lambda$ allows both the integral operator (when $\lambda=0$) and the summation operator (when $\lambda=1$), as well as quite a few other operators, to be encoded into one equation. Since then, Rota-Baxter algebra has found broad applications from combinatorics and number theory to classical Yang-Baxter equation and quantum field theory~\mcite{Bai,EGK,Guw,Gub,GZ,Ro1,Ro2,STS}.

Motivated by the close relationship between the differential and integral analysis as shown in the First Fundamental Theorem of Calculus, coordinated studied of differential algebra and Rota-Baxter algebra have emerged recently, beginning with the two simultaneously introduced concepts of a differential Rota-Baxter algebra and an integro-differential algebra.

The concept of a differential Rota-Baxter algebra~\mcite{GK3} is a simple coupling of a {\bf differential operator $d$ of weight $\lambda$}:
\begin{equation}
d(uv)=d(u)v+ud(v)+\lambda d(u)d(v), \ d(1)=0,
\mlabel{eq:diffl0}
\end{equation}
with a Rota-Baxter operator $P$ of the same weight by the abstraction of the First Fundamental Theorem of Calculus
\begin{equation}
d\circ P=\id,
\mlabel{eq:fft0}
\end{equation}
where $\id$ is the identity map. On the other hand, the concept of an {\bf integro-differential algebra}, first considered in the weight $0$ case in~\mcite{RR} and in the general weight case in~\mcite{GRR}, also takes into account the intertwining relationship of the two operators in the original definition of the integration by parts formula
\begin{equation}
P(d(u)P(v))=uP(v)-P(uv) - \lambda P(d(u)v).
\mlabel{eq:ibpl0}
\end{equation}
We note that Eq.~(\mref{eq:ibpl0}) implies Eq.~(\mref{eq:rb0}) at the presence of Eq.~(\mref{eq:fft0}) when $u$ is substituted by $P(u)$. Thus the variety of integro-differential algebras is the variety of differential Rota-Baxter algebras modulo extra conditions. See~\mcite{GRR} for further details.

As in the case of studying any algebraic structures, the free objects play an important role in the study of previous algebras. While the construction of free differential algebras is straightforward in terms of differential monomials, the construction of free Rota-Baxter algebras is more involved. In fact, there are three constructions in the commutative case, with the first one given by Rota~\mcite{Ro1} through an internal construction, and an external one given by Cartier~\mcite{Ca}. In~\mcite{GK1}, a construction is given by a generalization of the shuffle product, called the mixable shuffle product which is closely related to the quasi-shuffle product~\mcite{Ho} in the study of multiple zeta values.

By composing the construction of free differential algebras followed by that of the free Rota-Baxter algebras, free differential Rota-Baxter algebras were obtained in~\mcite{GK3}. Because of the more intimate relationship of the differential and Rota-Baxter operators in an integro-differential algebra, it is more challenging to construct free objects in the corresponding category even by the previous remark on the variety of integro-differential algebras, free integro-differential algebras are quotients of free differential Rota-Baxter algebras modulo the relation given by Eq.~(\mref{eq:ibpl0}). The first construction of free commutative integro-differential algebras was obtained in the recent paper~\mcite{GRR}. There the construction makes essential use of an equivalent formulation of the condition in Eq.~(\mref{eq:ibpl0}) for the integro-differential algebra.

\subsection{Gr\"obner-Shirshov bases}

In this paper, we apply the method of Gr\"obner-Shirshov bases to give another construction of the free commutative integro-differential algebras on a set.

The method of Gr\"obner bases or Gr\"obner-Shirshov bases originated from the work of Buchburger~\mcite{Bu} (for commutative polynomial algebras), Hironaka~\mcite{Hi} (for infinite series algebras) and Shirshov~\mcite{Sh} (for Lie algebras). It has since become a fundamental method in commutative algebra, algebraic geometry and computational algebra, and has been extended to many other algebraic structures, notably associative algebras~\mcite{Be,Bo}. In recent years, the method of Gr\"obner-Shirshov bases has been applied to a large number of algebraic structures to study problems on normal forms, word problems, rewriting systems, embedding theorems, extensions, growth functions and Hilbert series. See~\mcite{BCC,BCL,BLSZ} for further details.

This method also derives free objects in various categories, including the alternative constructions of free Rota-Baxter algebras and free differential Rota-Baxter algebras~\mcite{BCD,BCQ}. The basic idea is to prove a composition-diamond lemma that achieves a rewriting procedure to reduce any element to certain ``standard form". Then the set of elements in standard form is a basis of the free object.

We apply this method to construct a free commutative integro-differential algebra as the quotient of a free commutative differential Rota-Baxter algebra modulo the ``hybrid" integral by part formula in Eq.~(\mref{eq:ibpl0}). In order to do so, we would expect to first establish a Composition-Diamond Lemma for the free commutative differential Rota-Baxter algebra constructed in~\mcite{GK3}. We should then prove that the ideal generated by the defining relation of integro-differential algebras in Eq.~(\mref{eq:ibpl0}) has a Gr\"obner-Shirshov basis, thereby identifying a basis of a free commutative integro-differential algebra as a canonical subset of the known basis of the free commutative differential Rota-Baxter algebra. All these depend on the choice of a suitable monomial order on the set of the basis elements of the free commutative differential Rota-Baxter algebra. However a moment's thought reveals that such a monomial order does not exist for this algebra. To overcome this difficulty, we consider this algebra as a filtered algebra with respect to the order of derivation and study the filtration pieces first. Even there, we have to get along with a weakly monomial order which fortunately suffices for our applications. So we are able to adapt the above process of Gr\"obner-Shirshov bases and obtain a canonical basis for each of the filtration pieces. We then check that this process is compatible with the filtration structure, allowing us to put these canonical bases for the filtration pieces together to form a canonical basis for the entire free commutative integro-differential algebra. The following is our main theorem

\begin{theorem} $($=\text{Theorem~\mref{thm:gsb}}$)$
Let $X$ be a nonempty well-ordered set and $A:=\bfk\{X\}$. Let $\sha(\bfk\{X\})=\sha(\bfk[\Delta X])$, with the derivation $d$ and Rota-Baxter operator $P$, be the free commutative differential Rota-Baxter algebra of weight $\lambda$ on $X$. Let $I_{ID}$ be the differential Rota-Baxter ideal of $\sha(\bfk\{X\})$ generated by
$$S  := \{P(d(u) P(v))- uP(v)+ P(uv) + \lambda P(d(u) v) \mid u, v \in \sha(\bfk\{X\}) \}.
$$
Let $A_f$ be the submodule of $A=\bfk\{X\}$ spanned by functional monomials.
Then the composition
$$\sha(A)_f:=A \oplus \left(\bigoplus_{k\geq 0} A\otimes A_{f}^{\otimes k} \otimes A \right) \hookrightarrow  \sha(A) \to \sha(A) / I_{ID}$$
of the inclusion and the quotient map is a linear bijection.
Thus $\sha(A)_f$ gives an explicit construction of the free integro-differential algebra $\sha(A)/I_{ID}.$
\mlabel{thm:gsb0}
\end{theorem}

It is interesting to note that our approach of Gr\"obner-Shirshov bases gives a different construction of free commutative integro-differential algebras than those in~\mcite{GRR}. While the construction in~\mcite{GRR} has a transparent product formula, the construction here has a simple description as a submodule of the free differential Rota-Baxter algebra. By the uniqueness of the free objects, the two constructions yield isomorphic integro-differential algebras. Thus it would be interesting to compare the two constructions to reveal further the structure and properties of these free objects.

\subsection{Outline of the paper}
In Section~\mref{sec:prel}, we first introduce the algebraic structures that lead up to $\lambda$-integro-differential algebras and then recall the construction of free objects for these algebraic structures, in particular the free commutative Rota-Baxter algebras and the free commutative differential Rota-Baxter algebras. In Section~\mref{sec:mon}, we first give definitions related to differential Rota-Baxter monomials and then define a weakly monomial order on differential
Rota-Baxter monomials of order $n$. In Section~\mref{sec:cd}, we start with defining various kinds of
compositions and then establish the Composition-Diamond Lemma for the
$n$-th order free commutative differential Rota-Baxter algebra. In Section~\mref{sec:gs}, we consider a finite set $X$ and obtain a Gr\"obner-Shirshov basis for the defining ideal of a free commutative order $n$ integro-differential algebra on $X$ and thus obtain an explicitly defined basis for this free object. Then as mentioned above, we put the order $n$ pieces together as a direct system to obtain a basis for the free commutative integro-differential algebra on $X$. We then use a finiteness argument to treat the case when $X$ is any well-ordered set.

\section{Free commutative integro-differential algebras}
\mlabel{sec:prel}
We recall the definitions of algebras with various differential and integral operators and the constructions of the free objects in the corresponding categories.

\subsection{The definitions}
We recall the algebraic structures considered in this paper. We also introduce variations with bounded derivation order that will be needed later.
\begin{defn}
{\rm Let $\bfk$ be a unitary commutative ring. Let $\lambda\in \bfk$ be fixed.
\begin{enumerate}
\item A {\bf differential $\bfk$-algebra of weight $\lambda$} (also
    called a {\bf $\lambda$-differential $\bfk$-algebra}) is a unitary
    associative $\bfk$-algebra $R$ together with a linear operator
    $d \colon R\to R$ such that
\begin{equation}
d(1)=0,\ d(uv)=d(u)v+ud(v)+\lambda d(u)d(v) \text{ for all } u, v\in R.
\mlabel{eq:diffl}
\end{equation}
Such an algebra $(R,d)$ is said {\bf of order $n$}, where $n\geq 1$, if $d^n=0$.
\item A {\bf Rota-Baxter $\bfk$-algebra of weight $\lambda$} is an
    associative $\bfk$-algebra $R$ together with a linear operator
    $P\colon R\to R$ such that
\begin{equation}
P(u)P(v)=P(uP(v))+P(P(u)v)+\lambda P(uv) \text{ for all } u, v\in R.\mlabel{eq:rb}
\end{equation}

\item A {\bf differential Rota-Baxter k-algebra of weight $\lambda$} (also called a {\bf $\lambda$-differential
      Rota-Baxter $\bfk$-algebra}) is a differential $\bfk$-algebra
    $(R,d)$ of weight $\lambda$ and a Rota-Baxter operator $P$ of
    weight $\lambda$ such that
\begin{equation} d\circ P=\id.
\mlabel{eq:fft}
\end{equation}
\item An {\bf integro-differential $\bfk$-algebra of weight
      $\lambda$} (also called a {\bf $\lambda$-integro-differential
      $\bfk$-algebra}) is a differential $\bfk$-algebra $(R,d)$ of
    weight $\lambda$ with a linear operator $P\colon R \to R$ that satisfies Eq.~(\mref{eq:fft}) and such that
\begin{equation}
P(d(u)P(v))=uP(v)-P(uv) - \lambda P(d(u)v) \text{ for all } u, v\in R.
\mlabel{eq:ibpl}
\end{equation}
\end{enumerate}
}
\end{defn}

\subsection{Free differential Rota-Baxter algebras}

We first recall the construction of free commutative differential algebras and introduce their order $n$ variations.

\begin{theorem} Let $X$ be a set.
\begin{enumerate}
\item
Let $ \Delta X = \{ x^{(n)} \mid
x\in X, n\geq 0\}$ and let $\bfk\{X\}=\bfk[\Delta X]$ be the free commutative
algebra on the set $\Delta X$.  Define $d_X \colon
    \bfk[\Delta X]  \to \bfk[\Delta X]$ as follows. Let $w=u_1\cdots u_k,
    u_i\in \Delta X$, $1\leq i\leq k$, be a commutative word from the
    alphabet set $\Delta X $.  If $k=1$, so that $w=x^{(n)}\in
    \Delta X $, define $d_X(w)=x^{(n+1)}$. If $k>1$, recursively
    define
    \begin{equation}
      d_X(w)=d_X(u_1)u_2 \cdots u_k + u_1 d_X (u_2 \cdots u_k) + \lambda
      d_X(u_1)d_X(u_2 \cdots u_k).
      \mlabel{eq:prodind}
    \end{equation}
    Further define $d_X(1)=0$ and then extend $d_X$ to
    $\bfk[\Delta(X)]$ by linearity.  Then $(\bfk[\Delta X], d_X)$ is the
    free commutative differential algebra of weight $\lambda$ on the
    set $X$.  \mlabel{it:commfreediff}
\item
For a given $n\geq 1$, let $\Delta X^{(n+1)}:=\left\{x^{(k)}\,\big|\, x\in X, k\geq n+1\right\}$. Then $\bfk\{X\}\Delta X^{(n+1)}$ is the differential ideal $I_n$ of $\bfk\{X\}$ generated by the set $\{ x^{(n+1)}\,|\,x\in X\}$. The quotient $\bfk\{X\}/I_n$ has a canonical basis given by $\Delta_n X:=\{x^{(k)}\,|\, k\leq n\}$.
\mlabel{it:diffordn}
\end{enumerate}
\mlabel{thm:diff}
\end{theorem}
\begin{proof} Item~\mref{it:commfreediff} is from~\mcite{GK3} and Item~\mref{it:diffordn} is a direct consequence.
\end{proof}

For a set $Y$, let $C(Y)$ denote the free commutative monoid on $Y$. Thus elements in $\cm(Y)$ are commutative words, plus the identity $1$, from the alphabet set $Y$. Then $C(\Delta X)$ (resp. $C(\Delta_n X)$) is a linear basis of $\bfk[\Delta X]$ (resp. $\bfk[\Delta_n X]$).

We next recall the construction of free commutative Rota-Baxter
algebras in terms of mixable shuffles~\mcite{GK1,GK2}. The mixable shuffle product is shown to be the same as the quasi-shuffle product
of Hoffman~\mcite{EG,GZ,Ho}.  Let $A$ be a commutative
$\bfk$-algebra. Define

\begin{equation*}
  \sha (A)= \bigoplus_{k\geq 0} A^{\otimes (k+1)} = A\oplus A^{\otimes
    2}\oplus \cdots.
\end{equation*}
Let $\fraka =a_0\ot \cdots \ot a_m\in A^{\ot (m+1)}$ and
$\frakb=b_0\ot \cdots \ot b_n\in A^{\ot (n+1)}$. If $m=0$ or $n=0$,
define
\begin{equation}
  \fraka \shpr \frakb =\left \{\begin{array}{ll}
      (a_0b_0)\ot b_1\ot \cdots \ot b_n, & m=0, n>0,\\
      (a_0b_0)\ot a_1\ot \cdots \ot a_m, & m>0, n=0,\\
      a_0b_0, & m=n=0.
    \end{array} \right .
\end{equation}
If $m>0$ and $n>0$, inductively (on $m+n$) define
\begin{eqnarray}
  \fraka \shpr \frakb & = &
  (a_0b_0)\ot \Big(
  (a_1\ot a_2\ot \cdots \ot a_m) \shpr (1\ot b_1\ot \cdots \ot b_n) \notag \\
  &&
  \qquad \qquad +
  \; (1\ot a_1\ot \cdots \ot a_m) \shpr (b_1\ot \cdots \ot b_n) \mlabel{eq:shpr}\\
  && \qquad \qquad +
  \lambda\, (a_1\ot \cdots \ot a_m) \shpr (b_1\ot \cdots \ot b_n)\Big).
  \notag
\end{eqnarray}
Extending by additivity, we obtain a $\bfk$-bilinear map
\begin{equation*}
  \shpr: \sha (A) \times \sha (A) \rar \sha (A).
\end{equation*}
Alternatively,
\begin{equation*}
  \fraka\shpr \frakb=(a_0b_0)\otimes (\lbar{\fraka}
  \ssha_\lambda
  \lbar{\frakb}),
\end{equation*}
where~$\bar{\fraka} = a_1 \otimes \cdots \otimes a_m$, $\bar{\frakb}
= b_1 \otimes \cdots \otimes b_n$ and~$\ssha_\lambda$ is the mixable
shuffle (quasi-shuffle) product of weight
$\lambda$~\mcite{Gub,GK1,Ho}, which specializes to the shuffle product
${\tiny \ssha}$ when $\lambda=0$.

Define a $\bfk$-linear endomorphism $P_A$ on $\sha (A)$ by assigning
\[ P_A( x_0\otimes x_1\otimes \cdots \otimes x_n) =\bfone_A\otimes
x_0\otimes x_1\otimes \cdots\otimes x_n, \] for all $x_0\otimes
x_1\otimes \cdots\otimes x_n\in A^{\otimes (n+1)}$ and extending by
additivity.  Let $j_A\colon A\rar \sha (A)$ be the canonical
inclusion map.

\begin{theorem} $($\cite{GK1,GK2}$)$
\begin{enumerate}
\item
The pair $(\sha(A),P_A)$, together with the
  natural embedding $j_A\colon A\rightarrow \sha (A)$, is the free
  commutative Rota-Baxter $\bfk$-algebra on $A$ of weight $\lambda$.
  In other words, for any Rota-Baxter $\bfk$-algebra $(R,P)$ and any
  $\bfk$-algebra map $\varphi\colon A\rar R$, there exists a unique
  Rota-Baxter $\bfk$-algebra homomorphism $\tilde{\varphi}\colon (\sha
  (A),P_A)\rar (R,P)$ such that $\varphi = \tilde{\varphi} \circ
  j_A$ as $\bfk$-algebra homomorphisms.
\item Let $Y$ be a set and let $\bfk[Y]$ be the free
    commutative algebra on $Y$. The pair $(\sha(Y),P_Y):=(\sha (\bfk[Y]),
    P_{\bfk[Y]})$, together with the
  natural embedding $j_Y\colon Y\rightarrow \bfk[Y]\rightarrow \sha (\bfk[Y])$, is the free
  commutative Rota-Baxter $\bfk$-algebra of weight $\lambda$ on $Y$.
\end{enumerate}
  \mlabel{thm:shua}
\end{theorem}

Since $\shpr$ is compatible with the multiplication in $A$, we will often suppress the symbol $\shpr$ and simply denote $x y$ for $x\shpr y$ in $\sha (A)$, unless there is a danger of confusion.

A linear basis of $\sha(\bfk[Y])$ is given by
\begin{equation}
\frakB(Y):= \left\{ x_0\ot \cdots \ot x_k\,\big|\, x_i\in C(Y), 1\leq i\leq k, k\geq 0\right\},
\mlabel{eq:rbm}
\end{equation}
called the set of {\bf Rota-Baxter monomials} in $Y$.
The integer $\dep(x_0\ot \cdots \ot x_k):=k+1$ is called the {\bf depth} of $x_0\ot \cdots \ot x_k$.
To simplify notations, we also let $P$ denote $P_{\bfk[Y]}$. Then $1\ot u$ and $P(u)$ stand for the same element and will be be used as convenience in this paper.

We now put the differential and Rota-Baxter algebra structures together.
Let $(A, d_0)$ be a commutative differential $\bfk$-algebra of weight $\lambda$. Extend $d_0$ to $\sha(A)$ by
\begin{eqnarray*}
  \lefteqn{ d_A(x_0\otimes x_1\otimes\ldots\otimes x_k)}\\
  &=&
  d_0(x_0)\otimes x_1\otimes \ldots \otimes x_k + x_0x_1\otimes
  x_2 \otimes \ldots \otimes x_k +\lambda d_0(x_0) x_1\otimes x_2\otimes
  \ldots \otimes x_k, \quad k\geq 0.
\end{eqnarray*}

\begin{theorem}$($\cite{GK3}$)$
Let $X$ be a set and let $\bfk[\Delta X]$ be the free commutative differential algebra of weight $\lambda$ on $X$ in Theorem~\mref{thm:diff}.\mref{it:commfreediff}. The triple $(\sha (\bfk[\Delta X]), d_{\bfk[\Delta X]},
    P_{\bfk[\Delta X]})$, together with $j_X:X\to \Delta X\to \sha(\bfk[\Delta X])$, is the free commutative differential Rota-Baxter
    $\bfk$-algebra of weight $\lambda$ on $X$.
\delete{, as described by the
    following universal property: For any commutative differential
    Rota-Baxter $\bfk$-algebra $(R, d, P)$ of weight $\lambda$ and
    any set map $\varphi\colon X \to R$, there exists a unique
    $\lambda$-differential Rota-Baxter $\bfk$-algebra homomorphism
    $\tilde{\varphi}\colon (\sha (\bfk[\Delta X]), d_{\bfk[\Delta X]},
    P_{\bfk[\Delta X]})\rar (R, d, P)$ such that
    $\tilde{\varphi}\circ j_X = \varphi$.
    }
\mlabel{thm:freediffrb}
\end{theorem}

Apply the notations in Eq.~(\mref{eq:rbm}) to $Y:=\Delta X$. The set
\begin{equation}
\frakB(\Delta X):=\left\{ u_0\ot \cdots \ot u_k\,\big|\, u_i\in \cm(\Delta X), 0\leq i\leq k, k\geq 0\right\}
\mlabel{eq:drbm}
\end{equation}
is a $\bfk$-basis of the free commutative differential Rota-Baxter algebra $\sha(\Delta X)$, called the set of {\bf differential Rota-Baxter (DRB) monomials} on $X$.

Similarly  with $Y:=\Delta_n X, n\geq 1,$ $\calb(\Delta_n X)$ is a basis of $\sha(\Delta_n X)$ and is called the set of {\bf DRB monomials of order $n$} on $X$.
We note that in $\sha(\bfk[\Delta_n X])$, the property $d^{n+1}(u)=0$ only applies to $u\in X$, but not to tensors of length greater than two. For example, taking $n=1$, then
$d^2(x)=0$, but $d(1\ot x)= x$ and hence $d^2(1\ot x)=d(x)=x^{(1)}\neq 0$.

\subsection{Free commutative operated algebras}
We now construct the free commutative operated algebra on a set $X$ that has the free commutative (differential) Rota-Baxter algebra as a quotient. At the same time, the explicit construction $\sha(X)$ of free commutative Rota-Baxter algebra in Theorem~\mref{thm:shua} can be realized on a submodule of the free commutative operated algebra spanned by reduced words under a rewriting rule defined by the Rota-Baxter axiom.

This construction is parallel to that of the free (noncommutative) operated algebra on a set in~\mcite{BCQ,Gop,Gub,GSZ}. See~\mcite{Qiu} for the non-unitary case.

\begin{defn}
{\rm A {\bf commutative operated monoid with operator set $\Omega$} is a commutative monoid $G$ together with maps $\alpha_\omega:G\to G,\omega\in \Omega$. A homomorphism between commutative operated monoids $(G,\{\alpha_\omega\}_\omega)$ and $(H,\{\beta_\omega\}_\omega)$ is a monoid homomorphism $f:G\to H$ such that $f\circ \alpha_\omega=\beta_\omega\circ f$ for $\omega\in \Omega$.
}
\end{defn}
We next construct the free objects in the category of commutative operated monoids.

Fix a set $Y$. We define monoids $\frakC_n:=\frakC_n(Y)$ for $n\geq 0$ by a recursion.
First denote $\frakC_0:=\cm(Y)$.
Let $\lc \cm(Y)\rc_\omega:=\{\lc u\rc_\omega\,|\, u\in \cm(Y)\}, \omega\in \Omega,$ be disjoint sets in bijection with and disjoint from $\cm(Y)$.
Then define
$$\frakC_1:= \cm(Y\sqcup (\sqcup_{\omega\in \Omega} \lc \cm(Y)\rc_\omega)).$$
Note that elements in $\lc \cm(Y)\rc_\omega$ are only symbols indexed by elements in $\cm(Y)$. For example,
$\lc 1\rc_\omega$ is not the identity, but a new symbol.
The inclusion $Y\hookrightarrow Y\sqcup (\sqcup_{\omega\in \Omega}\lc\frakC_0\rc_\omega)$ induces a monomorphism
$i_{0,1}: \frakC_0=\cm(Y)\hookrightarrow \frakC_1=\cm(Y\sqcup (\sqcup_\omega\, \lc\frakC_0\rc_\omega))$ of free commutative monoids through which we identify $\frakC_0$ with its image in $\frakC_1$.
Inductively assume that $\frakC_{n-1}$ have been defined for $n\geq 2$ and that the embedding
$$i_{n-2,n-1}: \frakC_{n-2} \to \frakC_{n-1}$$
has been obtained. We then define
\begin{equation}
 \frakC_n:=\cm(Y\sqcup (\sqcup_\omega \lc\frakC_{n-1}\rc_\omega) ).
 \mlabel{eq:frakm}
 \end{equation}
We also have the injection
$$  \lc\frakC_{n-2}\rc_\omega \hookrightarrow
    \lc \frakC_{n-1} \rc_\omega, \ \omega\in \Omega.$$
Thus by the freeness of
$\frakC_{n-1}=\cm(Y\sqcup (\sqcup_\omega \lc\frakC_{n-2}\rc_\omega))$ as a free commutative monoid, we have
\begin{eqnarray*}
\frakC_{n-1} &=& \cm(Y\sqcup (\sqcup_\omega \lc\frakC_{n-2}\rc_\omega))\hookrightarrow
    \cm(Y\sqcup (\sqcup_\omega \lc \frakC_{n-1}\rc_\omega)) =\frakC_{n}.
\end{eqnarray*}
We finally define the commutative monoid
$$ \frakC(Y):=\bigcup_{n\geq 0}\frakC_n=\dirlim \frakC_n.$$
Elements in $\cmapm{Y}$ are called {\bf bracketed monomials} in $Y$. Defining
\begin{equation}
\lc\ \rc_\omega: \cmapm{Y}\to \cmapm{Y}, u\mapsto \lc u\rc_\omega, \ \omega\in \Omega,
\mlabel{eq:mapp}
\end{equation}
$(\cmapm{Y}, \{\lc\ \rc_\omega\}_\omega)$ is a commutative operated monoid and its linear span $(\bfk\cmapm{Y}, \lc\ \rc_\omega)$ is a commutative (unitary) operated $\bfk$-algebra.
\begin{prop}
Let $j_Y:Y \to \cmapm{Y}$ be the natural embedding.
\begin{enumerate}
\item
The triple $(\cmapm{Y},\{\lc\ \rc_\omega\}_\omega, j_Y)$ is the free commutative operated monoid on $Y$. More precisely, for any commutative operated monoid $G$ and set map $f:Y\to G$, there is a unique extension of $f$ to a homomorphism $\free{f}:\cmapm{Y}\to G$ of operated monoids.
\item
The triple $(\bfk\cmapm{Y},\{\lc\ \rc_\omega\}_\omega, j_Y)$ is the free commutative operated unitary $\bfk$-algebra on $Y$. More precisely, for any commutative $\bfk$-algebra $R$ and set map $f:Y\to R$, there is a unique extension of $f$ to a homomorphism $\free{f}:\bfk\cmapm{Y}\to R$ of operated $\bfk$-algebras.
\end{enumerate}
\mlabel{pp:freetm}
\end{prop}
\begin{proof}
We only need to show that $\cmapm{Y}$ is a free commutative operated monoid. The proof is similar to the noncommutative case~\mcite{Gop,Gub}, so we just give a sketch.

Let a commutative operated monoid $(G,\{\alpha_\omega\}_\omega)$ and a map $f:Y\to G$ be given. Then by the universal property of $\frakC_0:=C(Y)$, there is a unique monoid homomorphism $f_0:\frakC_0\to G$ extending $f$. Then $f_0$ extends uniquely to
$$f_1: \lc \frakC_0\rc_\omega \to G, \quad \lc u\rc_\omega \mapsto \alpha_\omega(f_0(u)), u\in \frakC_0,$$
such that $(f_1\circ \lc\ \rc_\omega)(u)= (\alpha_\omega\circ f_1)(u), \omega\in \Omega,$ when defined. We then further get a monoid homomorphism
$$ f_1: \frakC_1:=\cm (Y\sqcup (\sqcup_\omega \lc\frakC_0\rc_\omega)) \to G.$$
By induction on $n\geq 0$ we obtain a unique
$f_n: \frakC_n\to G, n\geq 0,$ compatible with the direct system, yielding the unique homomorphism
$\free{f}: \cmapm{Y}\to G$ of operated monoids.
\end{proof}

By the universal property of $\bfk\cmapm{Y}$, we obtain the following conclusion from general principles of universal algebra~\mcite{BN,Co}.
\begin{prop}
Let $\Omega=\{d, P\}$ and denote $d(u):=\lc u\rc_d, P(u):=\lc u\rc_P$\,.
Let $I_{DRB}$ be the operated ideal of $\bfk\cmapm{Y}$ generated by the set
$$\left\{\left . \begin{array}{l}
d(uv)-d(u)v-ud(v)-\lambda d(u)d(v),\\
 P(u)P(v)-P(uP(v)) -P(P(u) v) -\lambda P(uv),\\
 (d\circ P)(u)=u \end{array} \,\right|\, u, v\in \frak\cm(Y)\right\}.$$
Then the quotient operated algebra
$\bfk\cmapm{Y}/I_{DRB}$, with the quotient of the operator $d$ and $P$, is the free commutative differential Rota-Baxter algebra.
\mlabel{pp:freerb}
\end{prop}

Combining Proposition~\mref{pp:freerb} with Theorem~\mref{thm:shua}, we have

\begin{prop}\mlabel{pp:comp}
The natural embedding
$$ \sha(\bfk[\Delta X])\to \bfk\,\cmapm{\Delta X}, \quad
x_0\ot x_1\ot \cdots \ot x_k \mapsto x_0P(x_1P( \cdots P(x_k) \cdots ))$$
composed with the quotient map
$\rho: \bfk\,\frakC(\Delta X) \to \bfk\,\cmapm{\Delta X}/I_{DRB}$ gives a linear bijection (in fact, an isomorphism of differential Rota-Baxter algebras)
$$ \theta: \sha(\bfk[\Delta X])\to \bfk\,\cmapm{\Delta X}/I_{DRB}.$$
\end{prop}

Through $\theta$, we can identify the basis $\frakB(\Delta X)$ of $\sha(\bfk\Delta X)$ with its image in $\bfk\frakC(\Delta X)$:
\begin{equation}
u_0\ot u_1 \ot \cdots \ot u_k \leftrightarrow u_0\lc u_1\lc \cdots \lc u_k\rc\cdots \rc \rc \leftrightarrow u_0P (u_1 P(\cdots P(u_k)\cdots )).
\mlabel{eq:rbid}
\end{equation}
Thus we also use $P$ for $P_{\Delta X}$ on $\sha(\bfk[\Delta X])$ and $d^{\ell}(x)=x^{(\ell)}$ for $x\in X$ and $\ell\geq 0$.

As a consequence of Proposition~\mref{pp:comp}, we have
\begin{coro}\mlabel{co:comp}
Let $n\geq 1$. Let $I_{DRB,n}$ be the operated ideal of $\frakC(X)$ generated by $I_{DRB}$ together with the set $\{x^{(n+1)}=d^{n+1}(x)\, |\, x\in X\}$. The natural embedding
$$ \sha(\bfk[\Delta_n X])\to \bfk\cmapm{X}, \quad
x_0\ot x_1\ot \cdots \ot x_k \mapsto x_0P(x_1P( \cdots P(x_k) \cdots ))$$
composed with the quotient map
$\rho: \bfk\,\frakC(Y) \to \bfk\cmapm{Y}/I_{DRB,n}$ gives a linear bijection
$$ \theta_n: \sha(\bfk[\Delta_n X])\to \bfk\cmapm{X}/I_{DRB,n}.$$
\end{coro}

\begin{proof}
The map $\theta_n$ is obtained by starting from the isomorphism $\theta: \sha(\bfk[\Delta X]) \cong \bfk\frakC(X)/I_{DRB}$ and then taking the quotients of both the domain and range by the operated ideal generated by $d^{n+1}(x), x\in X$. Since $\theta$ restricted to the identity on $X$. The corollary follows.
\end{proof}

Define the {\bf reduction map}
\begin{equation}
\red:=\red_n:=\theta_n^{-1}\circ \rho: \bfk\,\cmapm{X} \to \bfk\cmapm{Y}/I_{DRB,n} \to \sha(\bfk[\Delta_n X]).
\mlabel{eq:red}
\end{equation}
It reduces any bracketed monomial to a DRB monomial. For example, if $u, v\in \cm(X)$, then
$$\red(\lc u\rc \lc v\rc)=1\ot u\ot v + 1\ot v\ot u +\lambda \ot uv.$$

\section{Weakly monomial order}
\mlabel{sec:mon}
In this section, we will give a weak form of the monomial order on filtered pieces of the set of differential Rota-Baxter monomials. It will be sufficient for us to establish the composition-diamond lemma for integro-differential algebras.

Let $Y$ be a set with well order $\leq_Y$. Define the {\bf length-lexicographic order} $\leq_{Y,\lex}^*$ on the free monoid $M(Y)$ by
\begin{equation}
u<_{Y,\lex}^* v \Leftrightarrow \left\{\begin{array}{l} \ell< m, \\
\text{or } \ell=m \text{ and } \exists 1\leq i_0\leq \ell \text{ such that } u_i=v_i \text{ for } 1\leq i<i_0 \text{ and } u_{i_0}<v_{i_0}, \end{array}\right.
\mlabel{eq:lex}
\end{equation}
where $u=u_1\cdots u_\ell$ and $v=v_1\cdots v_m$ with $u_i\in Y, 1\leq i\leq \ell, v_j\in Y, 1\leq j\leq m, m, n\geq 1$.
It is well-known~\mcite{BN} that $\leq_{Y,\lex}^*$ is still a well order.
An element $1\neq u$ of the free commutative monoid $\cm(Y)$ can be uniquely expressed as
\begin{equation}
u = u_0^{j_0} \cdots u_k^{j_k}, \text{ where } u_0,\cdots,u_k\in
Y, j_0,\cdots, j_k \in \mathbb{Z}_{\geq 1} \text{ and } u_0
> \cdots > u_k. \mlabel{eq30}
\end{equation}
This expression is called the {\bf standard form} of $u$. If $k= -1$, we take $u\in \bfk$ by convention.

Any $1\neq u\in \cm(Y)$ can also be expressed uniquely as
$$u=u_1\cdots u_\ell,\ u_1\geq u_2\geq \cdots \geq u_\ell \in Y.$$
With this notation, $\cm(Y)$ can be identified with a subset of the free monoid $M(Y)$ on $Y$. Then the well order $<_{Y,\lex}^*$ on $M(Y)$ restricts to a well order on $\cm(Y)$.

\begin{lemma}\label{mul diff}
Let $(Y,\leq_Y)$ is a well-ordered set and $u,v\in \cm(Y)$. If $u<v$, then $uw \leq_{Y,\lex}^* vw$ for $w\in \cm(Y)$. \mlabel{lemma:mul diff}
\end{lemma}

\begin{proof}
Such a result is well-known for free noncommutative monoid. The proof for the commutative case is different and we sketch a proof for completeness.

From the standard decomposition of $u\in \cm(Y)$ in Eq.~(\mref{eq30}), $u$ can be expressed uniquely as a function
\begin{equation} f:=f_u: Y\to \ZZ_{\geq 0}, f_u(y)=\left\{\begin{array}{ll} j_i, & y=u_i, 1\leq i\leq k, \\ 0, & \text{otherwise}. \end{array}\right .
\mlabel{eq:cmf}
\end{equation}
Thus $\cm(Y)$ can be identified with
$$ \calf:=\{f: Y\to \ZZ_{\geq 0}\,|\, \supp (f):=Y\backslash f^{-1}(0) \text{ is finite } \}$$
with $1\in \cm(Y)$ corresponding to $f_1\equiv 0$. Denote $\deg(f):=\sum_{y\in Y} f(y)$. Under this identification, the order $\leq_{Y,
\lex}^*$ on $\cm(Y)$ is identified with the order $\leq$ on $\calf$ defined by
\begin{equation} \label{eq:cmford}
f< g \Leftrightarrow \left\{\begin{array}{l} \deg(f)<\deg(g)\\
\text{or } \deg(f)=\deg(g) \text{ and } \exists y_0\in Y \text{ such that } f(y)=g(y) \text{ for } y<y_0 \text{ and } f(y_0)< g(y_0).
\end{array} \right .
\end{equation}

Let $u, v, w\in \cm(Y)$ be given. We apply the identification of $u, v, w$ with $f_u, f_v, f_w\in \calf$ given in Eq~(\mref{eq:cmf}). We note that $f_{uw}=f_u +f_w$ and $f_{vw}=f_v+f_w$. Thus we have
$$ \deg(f_{uw})=\deg(f_u)+\deg(f_w), \deg(f_{vw})=\deg(f_v)+\deg(f_w), \text{ and } f_u(y)<f_v(y) \Leftrightarrow f_{uw}(y)<f_{vw}(y).$$
Then it follows that $f_u<f_v$ if and only if $f_{uw}<f_{vw}$. This proves the lemma.
\end{proof}

For a set $X$, recall that $\Delta X =\{x^{(k)}
\mid x\in X, k\geq 0\}$ and $\Delta_n X :=\{x^{(k)} \mid x\in X, 0\leq k\leq n\}$
for $n\geq 0$. Then $\cm(\Delta_n X), n\geq 0,$ define an increasing filtration on $\cm(\Delta X)$ and hence give a filtration $\calb(\Delta_n X)\subseteq \calb(\Delta X)$.
Elements of $\calb(\Delta_n X)$ are called
{\bf DRB monomials of order $n$}.

\begin{defn}
{\rm Let $X$ be a set, $\star$ a symbol not in $X$ and $\Delta_n
X^\star := \Delta_n (X\cup \{\star\})$.
\begin{enumerate}
\item
By a {\bf
$\star$-DRB monomial on $\Delta_n X$},
we mean any expression in $\calb(\Delta_n X^\star)$ with exactly one
occurrence of $\star$. The set of all $\star$-DRB monomials on $\Delta_n X$ is denoted by
$\calb^\star(\Delta_n X)$.
\item
For $q\in \calb^\star(\Delta_n X)$ and
$u\in \calb(\Delta_n X)$, we define
$$q|_u := q|_{\star \mapsto u}$$
to be the bracketed monomial in $\cmapm{\Delta_n X}$ obtained by replacing the letter $\star$ in $q$ by
$u$, and call $q|_u$ a {\bf $u$-monomial on $\Delta_n X$}.
\item
Further, for $s=\sum_i c_i u_i \in \bfk
\calb(\Delta_n X)$, where $c_i\in \bfk$, $u_i\in \calb(\Delta_n X)$
and $q\in \calb^\star(\Delta_n X)$, we define
$$q|_s := \sum_i c_i q|_{u_i},$$
which is in $\bfk\,\cmapm{\Delta_n X}$.
\end{enumerate}
}
\end{defn}
We note that a $\star$-DRB monomial $q$ is a DRB monomial in $\Delta_n X^\star$ while its substitution $q|_u$ might not be a DRB monomials. For example, for $q=P(x_1)\star\in \calb(\Delta_n X^\star)$ and $u=P(x_2)\in \calb(\Delta_n X)$ where $x_1, x_2\in X$, the $u$-monomial $q|_u=P(x_1)P(x_2)$ is no longer in $\calb(\Delta_n X)$.

\begin{lemma} \label{operator ideal}\mlabel{lemma:operator ideal}
Let $S$ be a subset of $\bfk\frakC(\Delta_n(X))$ and $\mathrm{Id(S)}$
be the operated ideal of $\bfk \frakC(\Delta_n(X))$ generated by $S$.
Then
$$\mathrm{Id(S)} = \left\{ \sum_{i=1}^{k} c_i q_i | _{s_i} \,\Big|\,
c_i\in \bfk,  q_i\in \frakC^{\star}(\Delta_n X), s_i\in S, 1\leq
i\leq k, k\geq 1 \right\}.$$
\end{lemma}

\begin{proof}
It is easy to see that the right hand side is contained in the left
side. On the other hand, the right hand side is already an operated
ideal of $\bfk \frakC(\Delta_n(X))$ containing $S$.
\end{proof}

\begin{defn}
{\rm If $q = p|_{d^\ell(\star)}$ for some $p\in \calb^\star(\Delta_n
X)$ and $\ell\in \mathbb{Z}_{\geq 1}$, then we call $q$ a {\bf
type I $\star$-DRB monomial}. Let $\calb_{I}^\star(\Delta_n X)$ denote the set of type I $\star$-DRB monomials on
$\Delta_n X$ and call
$$\calb_{II}^\star(\Delta_n X) :=
\calb^\star(\Delta_n X) \setminus \calb_{I}^\star(\Delta_n X)$$
the set of {\bf type II $\star$-DRB monomials}. }
\end{defn}

\begin{lemma}\mlabel{threecases}
Any element $q\in \calb^\star(\Delta_n X)$ is one of the following three forms
\begin{enumerate}
 \item  $q\in \calb_{I}^\star(\Delta_n X)$, or
 \item $q = s\star t$ with $s\in \cm(\Delta_n X)$ and $t\in
\calb(\Delta_n X)$, or
\item $q = sP(p)$ for some $s\in
\cm(\Delta_n X)$ and $p \in \calb^\star_{II}(\Delta_n X)$.
\end{enumerate}
\end{lemma}

\begin{proof}
Any element $q\in \calb^{\star}(\Delta_n X)$ is of the form $u_0\ot u_1\ot \cdots \ot u_k$ with $u_i\in \cm(\Delta_n X), 1\leq i\leq k,$ except a unique $u_i$ which is in $\cm(\Delta_n X^\star)$ with exactly one occurrence of $\star$. In turn, this unique $u_i\in \cm(\Delta_n X^\star)$ is of the form $u_{i1}\cdots u_{im}$ with $u_{ij}\in \Delta_n X, 1\leq j\leq m,$ except a unique $u_{ij}$ which is in $\Delta_nX^\star$ with exactly one occurrence of $\star$. Thus this unique $u_{ij}\in \Delta_nX^\star$ is of the for $d^\ell(\star)$ for some $\ell\geq 0$. If $\ell\geq 1$, then $q$ is of type I. If $\ell=0$, then $d^{\ell}(\star)=\star$. So if $i=0$, namely this $\star$ is in $u_0$, then $q=(u_{01}\cdots u_{0(j-1)}\star u_{0(j+1)} \cdots u_{0m})\ot u_1\ot \cdots u_k$ is of the form $s\star t$ with $s=u_{01}\cdots u_{0(j-1)}u_{0(j+1)}\cdots u_{0m}$ and $t=1\ot u_2\cdots u_k$. If $i\geq 1$, then $q=sP(p)$, where $p:=u_2\ot\cdots u_k \in
\calb^\star_{II}(\Delta_n X)$. This proves the lemma.
\end{proof}

\begin{defn}{\rm
Let $X$ be a set, $\star_1$, $\star_2$ two distinct symbols not in
$X$ and $\Delta_n X^{\star_1, \star_2} := \Delta_n (X\cup
\{\star_1,\star_2\})$. We define a {\bf
$(\star_1,\star_2)$-DRB monomial on
$\Delta_n X$ } to be an expression in $\calb(\Delta_n
X^{\star_1,\star_2})$ with exactly one occurrence of $\star_1$ and
exactly one occurrence of $\star_2$. The set of all $(\star_1,
\star_2)$-DRB monomials on $\Delta_n X$
is denoted by $\calb^{\star_1, \star_2}(\Delta_n X)$. For $q\in
\calb^{\star_1, \star_2}(\Delta_n X)$ and $u_1, u_2\in \bfk
\calb(\Delta_n X)$, we define
$$q|_{u_1,u_2} := q|_{\star_1 \mapsto u_1, \star_2 \mapsto u_2}$$
to be the bracketed monomial obtained by
replacing the letter $\star_1$ (resp. $\star_2$) in $q$ by $u_1$
(resp. $u_2$) and call it a {\bf $(u_1,u_2)$-bracketed monomial on $\Delta_n X$ }. }
\end{defn}

A $(u_1,u_2)$-DRB monomial on $\Delta_n
X$ can also be recursively defined by
\begin{equation} \label{eq12}
q|_{u_1,u_2} := (q^{\star_1}|_{u_1})|_{u_2},
\end{equation}
where $q^{\star_1}$ is $q$ when $q$ is regarded as a
$\star_1$-DRB monomial on the set
$\Delta_n X^{\star_2}$. Then $q^{\star_1}|_{u_1}$ is in
$\calb^{\star_2}(\Delta_n X)$. Similarly, we have
\begin{equation} \label{eq13}
q|_{u_1,u_2} := (q^{\star_2}|_{u_2})|_{u_1}.
\end{equation}

Let $X$ be a well-ordered set and let $Y=\Delta X$. Let $n\geq 0$ be given. For $x_0^{(i_0)}, x_1^{(i_1)}\in \Delta X$ (resp. $\Delta_n X$) with $x_0, x_1\in X$, define
\begin{equation}
x_0^{(i_0)} \leq x_1^{(i_1)} \left(\text{resp.} x_0^{(i_0)}\leq_n x_1^{(i_1)}\right) \Leftrightarrow (x_0,-i_0) \leq (x_1, -i_1) \quad \text{
lexicographically}.
\mlabel{eq:difford}
\end{equation}
For example $x^{(2)} < x^{(1)}< x$. Also, $x_1<x_2$ implies
$x_1^{(2)} < x_2^{(2)}$. Then by~\mcite{BN}, the order $\leq_n$ is a well order on $\Delta_n X$ and hence is extended to a well order on $\cm(\Delta_n X)$ by Eq.~(\mref{eq:lex}) which we still denote by $\leq_n$.

We next extend the well order $\leq_n$ on $\cm(\Delta_n X)$ defined above to $\calb(\Delta_n X)$. Note that $$\calb(\Delta_n X)= \{u_0\ot u_1\ot \cdots u_k\,|\, u_i\in \cm(\Delta_n X), 1\leq i\leq k, k\geq 0\}=\sqcup_{k\geq 1} \cm(\Delta_nX)^{\ot k} $$
can be identified with the free semigroup on the set $\cm(\Delta_nX)$. Thus the well order $\leq_n$ on $\cm(\Delta_nX)$ extends to a well order $\leq_{n,\lex}^*$~\mcite{BN} which we will still denote by $\leq_n$ for simplicity. More precisely, for any $u=u_0\otimes \cdots \otimes u_k \in
\cm(\Delta_n X)^{\ot (k+1)}$ and $v = v_0 \otimes \cdots
\otimes v_\ell \in \cm(\Delta_n X)^{\ot (\ell+1)}$, define
\begin{equation}
u \leq_{n} v \text { if }(k+1, u_0, \cdots,u_k) \leq
(\ell+1, v_0, \cdots,v_\ell) \text{
lexicographically}. \mlabel{eq7}
\end{equation}

This is the order on $\frakB(\Delta_nX)$ that we will consider in this paper.

\begin{defn}
Let $\leq_n$ be the well order on $\frakB(\Delta_nX)$ defined in Eq.~(\mref{eq7}).
Let $q\in \calb^\star(\Delta_n X)$ and $s\in
\mathbf{k}\calb(\Delta_n X)$.
\begin{enumerate}
\item
For any $0\neq f\in \bfk \calb(\Delta_n X)$, let $\lbar{f}$ denote the leading term of $f$:
$f = c \overline{f} + \sum_{i} c_iu_i$, where $0\neq  c, c_i\in
\bfk$, $u_i\in \calb(\Delta_n X)$, $u_i< \overline{f}$. $f$ is called {\bf monic} if $c=1$.
\item
Denote
$$\overline{q |_s} := \overline{\red(q|_{s}}),$$
where $\red: \bfk\frakC(\Delta_nX)\to \sha(\Delta_nX)=\bfk \calb(\Delta_n X)$ is the reduction map in Eq.~(\mref{eq:red}).
\item
The element $q|_s\in \bfk\,\cmapm{\Delta_nX}$ is called {\bf normal} if
$q|_{\lbar{s}}$ is in $\calb(\Delta_n X)$. In other words, if $\red(q|_{\lbar{s}}) = q|_{\lbar{s}}$.
\end{enumerate}
\mlabel{normaldef}
\end{defn}

\begin{remark} \begin{enumerate}
\item By definition, $q|_s$ is normal if and only if
$q|_{\lbar{s}}$ is normal if and only if the $\lbar{s}$-DRB monomial $q|_{\lbar{s}}$ is already a DRB monomial, that is, no further reduction in $\sha(\Delta_n X)$ is possible.
\item
Examples of not normal (abnormal) $s$-DRB monomials are
\begin{enumerate}
\item $q=\star P(x)$ and $\bar{s}=P(x)$, giving $q|_{s}=P(x)P(x)$ which is reduced to $P(xP(y))+P(P(x)y)+\lambda P(xy)$ in $\sha(\Delta_n X)$;
\item $q=d(\star)$ and $\bar{s}=P(x)$, giving $q|_{\bar{s}}=d(P(x))$ which is reduced to $x$ in $\sha(\Delta_n X)$;
\item $q=d(\star)$ and $\bar{s}=x^2$, giving $q|_{\bar{s}}=d(x^2)$ which is reduced to $2xx^{(1)}+\lambda (x^{(1)})^2$ in $\sha(\Delta_n X)$;
\item $q=d^n(\star)$ and $\bar{s}=d(x)$, giving $q|_{\bar{s}}=d^{n+1}(s)$ which is reduced to $0$ in $\sha(\Delta_n X)$.
\end{enumerate}
\end{enumerate}
\end{remark}

\begin{defn}\label{defweakmonomial}
A {\bf weakly monomial order} on $\calb(\Delta_n X)$ is a well order
$\geq$ satisfying the following condition:
$$
\text{for } u, v\in \calb(\Delta_n X), u > v\, \Rightarrow
\overline{q|_u} > \overline{q|_v} \text{ if either } q \in \calb^{\star}_{II}(\Delta_n X), \text{ or } q \in
\calb^{\star}_{I}(\Delta_n X) \text{ and } q|_u  \text{ is normal}.
$$
\end{defn}

We shall prove that the order defined in Eq.~(\ref{eq7}) is a weakly
monomial order on $\calb(\Delta_n X)$. We need the following lemmas.

\begin{lemma}
Let $\ell\geq 1$ and $s\in
\calb(\Delta_n X)$. Then $d^\ell(\star)|_{s}$ is normal if and only if $s\in
\Delta_{n-\ell} X$. \mlabel{lemma:diffnormal}
\end{lemma}

\begin{proof}
If $\lbar{s} \in \Delta_{n-\ell} X$, then $d^{\ell}(\lbar{s})$ is in $\Delta_nX$ and hence
$d^\ell(\star)|_{s}$ is normal. Conversely, if $\lbar{s} \notin
\Delta_{n-\ell} X$, then either $\mathrm{dep}(\lbar{s}) \geq 2$, or
$\mathrm{dep}(\lbar{s}) =1$ and $\mathrm{deg}_{\Delta_n X}(u)\geq
2$, or $\lbar{s} \in \Delta_n X \setminus \Delta_{n-\ell} X$. In all
these cases, $d^\ell(\star)|_{s}$ is not normal.
\end{proof}

\begin{lemma}
\label{leading diffRB} Let $\leq_n$ be the order defined in Eq.~(\mref{eq7}). Let $u, v\in \calb(\Delta_n X)$ and $\ell\in \mathbb{Z}_{\geq 1}$. If $u >_n v$ and $d^\ell(\star)|_u$ is normal,
then $\lbar{d^\ell(u)} >_n \lbar{d^\ell(v)}$. \mlabel{lemma: leading
diffRB}
\end{lemma}

\begin{proof}
We prove the result by induction on $\ell$. We first consider $\ell
=1$ and prove $\lbar{d(u)} >_n \lbar{d(v)}$. Since $d(\star)|_u$ is normal, we have $u=x_1^{(i_1)}\in \Delta_{n-1} X$ by Lemma
\ref{lemma:diffnormal}. Since $u >_n v$, by the definition of $>_n$, we have we have $v=x_2^{(i_2)}\in \Delta_n X$ with either $x_1>x_2$ or $x_1=x_2$ and $i_1<i_2$. So $\lbar{d(u)} >_n \lbar{d(v)}$.

Next, suppose the result holds for $1\leq m <\ell$. Then by the induction hypothesis, we have
$$\lbar{d^{\ell}(u)} = \lbar{d(d^{\ell-1}(u))} = \lbar{d(\lbar{d^{\ell-1}(u)})} >_n \lbar{d(\lbar{d^{\ell-1}(v)})}=
\lbar{d(d^{\ell-1}(v))} = \lbar{d^{\ell}(v)}.$$
\end{proof}

\begin{prop}
The order $\leq_n$ defined in Eq.~(\ref{eq7}) is a weakly monomial order.
\mlabel{weakmonomial}
\end{prop}

\begin{proof}
Let $u,v\in \calb(\Delta_n X)$ with $u>_nv$ and $q\in
\calb^\star(\Delta_n X)$. By Lemma \mref{threecases} we have the
following three cases to consider.
\smallskip

\noindent
{\bf Case 1.} Consider $q = s\star t$ where $s\in \cm(\Delta_n X)$ and $t\in \calb(\Delta_n X)$. Note that $\calb(\Delta_n X) =
\cm(\Delta_n X)\sqcup \cm(\Delta X)P(\calb(\Delta_n X))$. We consider the following four subcases depending on $t$ or $u$ in $\cm(\Delta_n X)$ or $\cm(\Delta X)P(\calb(\Delta_n X))$.
\smallskip

\noindent
{\bf Subcase 1.1.} Let $t, u\in \cm(\Delta_n X)$. Since $u>_n v$, we have
that $v\in \cm(\Delta_n X)$ and so by Lemma \mref{lemma:mul diff},
$\lbar{q|_{u}} = sut >_n svt = \lbar{q|_{v}}$.
\smallskip

\noindent
{\bf Subcase 1.2.} Let $t\in \cm(\Delta_n X)P(\calb(\Delta_n X))$ and
$u\in \cm(\Delta_n X)$. Let $t =  t_0 \otimes \cdots \otimes t_m$ with
$m\geq 1$. Since $u >_n v$, we have $v\in \cm(\Delta_n X)$. By Lemma
\mref{lemma:mul diff} and Eq.~(\mref{eq7}), we have
$$(m+1, sut_0, t_1, \cdots, t_m) > (m+1, svt_0, t_1, \cdots, t_m) \quad \text{ lexicographically}.$$
So $\lbar{q|_{u}} = (sut_0) \otimes t_1\ot \cdots \otimes t_m >_n (svt_0)
\otimes t_1\ot \cdots \otimes t_m = \lbar{q|_{v}}$.
\smallskip

\noindent
{\bf Subcase 1.3.} Let $t\in \cm(\Delta_n X)$ and $u \in \cm(\Delta_n
X)P(\calb(\Delta_n X))$. Let $u =  u_0 \otimes \cdots \otimes u_k$
with $k\geq 1$. If $v \in \cm(\Delta_n X)$, it is obvious that
$$ \lbar{q|_{u}} = (stu_0) \otimes \cdots
\otimes u_k >_n \lbar{q|_{v}} = svt.$$ If $v\in \cm(\Delta_n
X)P(\calb(\Delta_n X))$, let $v = v_0 \otimes \cdots \otimes v_m$
with $m\geq 1$. Then $\lbar{q|_{v}} = (stv_0) \otimes \cdots \otimes
v_m$. Since $u>_n v$, by Eq.~(\ref{eq7}), we have that $$(k+1,
u_0,\cdots, u_k) > (m+1, v_0,\cdots, v_m)\quad \text{
lexicographically}.$$ By Lemma \ref{mul diff}, it follows that
$$(k+1, stu_0,u_1,\cdots, u_k)> (m+1, stv_0,
v_1,\cdots,v_m) \quad \text{ lexicographically},$$ that is,
$\lbar{q|_{u}} >_n \lbar{q|_{v}}$.
\smallskip

\noindent
{\bf Subcase 1.4.} Let $t,u\in \cm(\Delta_n X)P(\calb(\Delta_n X))$. Let
$t =t_1\otimes \tilde{t} = t_0P(\tilde{t})$ and $u = u_0\otimes
\tilde{u} =u_0P(\tilde{u})$, where $t_0,u_0\in \cm(\Delta_n X),
\tilde{t},\tilde{u} \in \calb(\Delta_n X)$. If $v\in \cm(\Delta_n X)$,
then $$\lbar{q|_{v}} = (svt_0) P(\tilde{t})\quad \text{ and }\quad
\lbar{q|_{u}} = \lbar{st_0u_0 P(\tilde{t})P(\tilde{u})} = st_0u_0
\lbar{P(\tilde{t})P(\tilde{u})}.$$ Thus $\mathrm{dep}( \lbar{q|_u} )
> \mathrm{dep}( \lbar{ q|_v } )$ and so $\lbar{q|_{u}} >_n
\lbar{q|_{v}}$. If $v\in \cm(\Delta X)P(\calb(\Delta_n X))$, let $v =
v_0 \otimes \tilde{v} =v_0 P(\tilde{v})$. Then $ \lbar{q|_{v}} =
\lbar{st_0 v_0 P(\tilde{t})P(\tilde{v})} = s t_0 v_0
\lbar{P(\tilde{t})P(\tilde{v})}$. Since $u >_n v$, we have
$\mathrm{dep}(\tilde{u}) +1 = \mathrm{dep}(u) \geq \mathrm{dep}(v) =
\mathrm{dep}(\tilde{v}) +1$ by Eq.~(\ref{eq7}) and so
$\mathrm{dep}(\tilde{u})\geq \mathrm{dep}(\tilde{v})$. If
$\mathrm{dep}(\tilde{u}) > \mathrm{dep}(\tilde{v})$, then
$\mathrm{dep}(\lbar{q|_{u}}) > \mathrm{dep}(\lbar{q|_{v}})$ and so
$\lbar{q|_{u}} >_n \lbar{q|_{v}}$. Suppose $\mathrm{dep}(\tilde{u}) =
\mathrm{dep}(\tilde{v})$. Then $ \mathrm{dep}(\lbar{q|_{u}}) =
\mathrm{dep}(\lbar{q|_{v}})$. If $u_0 >_n v_0$, then $st_0u_0
>_n st_0v_0$ by Lemma \ref{mul diff} and so $\lbar{q|_{u}} >_n
\lbar{q|_{v}}$ by Eq.~(\ref{eq7}). We are left to consider the case
$\mathrm{dep}(\tilde{u}) = \mathrm{dep}(\tilde{v})$ and $u_0 = v_0$.
In this case, since $u >_n v$, we have that $\tilde{u} >_n \tilde{v}$.
If $ \tilde{u} >_n \tilde{v} \geq \tilde{t}$, then
\begin{align*}
&\lbar{q|_{u}} =(st_0 u_0) \lbar{P(\tilde{t})P(\tilde{u})} = (st_0
u_0)P(\tilde{u}P(\tilde{t})) = (st_0 u_0) \otimes\tilde{u}
\otimes \tilde{t} \\
>_n &(st_0 v_0) \otimes\tilde{v} \otimes \tilde{t} = (st_0v_0)
P(\tilde{v}P(\tilde{t})) = (st_0 v_0)\lbar{P(\tilde{t})P(\tilde{v})}
= \lbar{q|_{v}}.
\end{align*}
If $ \tilde{t} \geq \tilde{u}>_n \tilde{v}$, then
\begin{align*}
&\lbar{q|_{u}} = (st_0u_0) \lbar{P(\tilde{t})P(\tilde{u})} =
(st_0u_0) P(\tilde{t}P(\tilde{u})) = (st_0u_0) \otimes\tilde{t}
\otimes \tilde{u} \\
>_n &(st_0v_0 )\otimes\tilde{t} \otimes \tilde{v} = (st_0u_0)
P(\tilde{t}P(\tilde{v})) = (st_0u_0)\lbar{P(\tilde{t})P(\tilde{v})}=
\lbar{q|_{v}}.
\end{align*}
If $\tilde{u}>_n \tilde{t} >_n \tilde{v}$, then
\begin{align*}
&\lbar{q|_{u}} = (st_0u_0) \lbar{P(\tilde{t})P(\tilde{u})} =
(st_0u_0) P(\tilde{u}P(\tilde{t})) = (st_0u_0) \otimes\tilde{u}
\otimes \tilde{t} \\
>_n &(st_0v_0 )\otimes\tilde{t} \otimes \tilde{v} = (st_0u_0)
P(\tilde{t}P(\tilde{v})) = (st_0
u_0)\lbar{P(\tilde{t})P(\tilde{v})}= \lbar{q|_{v}}.
\end{align*}

\noindent
{\bf Case 2.} Consider $q = sP(p)$ for some $s\in \cm(\Delta_n X)$ and
$p \in \calb^\star(\Delta_n X)$. This case can be verified by
induction on $\mathrm{dep}(q)$ and the fact that, for $u,v \in \calb(\Delta_n X)$, $u >_n v$ implies $P(u) >_n P(v).$
\smallskip

\noindent
{\bf Case 3.} Consider $q\in \calb_I^\star(\Delta_n X)$. Then $q =
p|_{d^{\ell}(\star)}$ for some $p\in \calb^\star(\Delta_n X)$ and
$\ell\in \mathbb{Z}_{\geq 1}$. Take such $\ell$
maximal so that $p\in \calb_{II}^\star(\Delta_n X)$. We need to show that if $u>_nv$ and $q|_u$ is
normal, then $\lbar{q|_{u}} >_n \lbar{q|_{v}}$. But if $q|_u$ is normal then $d^{\ell}(\star)|_u$ is normal. Then by Lemma
\ref{leading diffRB}, we have $\lbar{d^\ell(u)} >_n
\lbar{d^\ell(v)}$. Then by Cases 1 and 2, we have $\lbar{q|_{u}} =
\lbar{p|_{\lbar{d^\ell(u)}}}
>_n\lbar{p|_{\lbar{d^\ell(v)}}} = \lbar{q|_{v}}$. This completes the proof. \end{proof}

We give the following consequences of Proposition~\mref{weakmonomial} to be applied in Section~\mref{sec:cd}.
\begin{lemma}
Let $q\in \calb^\star(\Delta_n X)$ and $s\in
\mathbf{k}\calb(\Delta_n X)$ be monic. If $q|_s$ is normal, then
$\lbar{q|_s} = q|_{\lbar{s}}$. \mlabel{normalequiv}
\end{lemma}

\begin{proof} Let $s = \lbar{s} + \sum_i c_i s_i$ with $0\neq c_i\in \bfk$ and $s_i < \lbar{s}$.
Then $q|_s = q|_{\lbar{s}} + \sum_i c_i q|_{s_i}$. Since $q|_s$ is
normal, it follows that $q|_{\lbar{s}} \in \calb(\Delta_n X)$ and so
$\lbar{q|_{\lbar{s}}} = q|_{\lbar{s}}$. We have the following two
cases to consider.

\noindent
{\bf Case I.} $q\in \calb_{II}^{\star}(\Delta_n X)$. Then
$\lbar{q|_{s_i}} < \lbar{q|_{\lbar{s}}} = q|_{\lbar{s}}$ by
Definition \ref{defweakmonomial} and Proposition~\ref{weakmonomial}. Hence
$\lbar{q|_s} = \lbar{q|_{\lbar{s}}}=q|_{\lbar{s}}$.

\noindent
{\bf Case II.} $q\in \calb_{I}^{\star}(\Delta_n X)$. Then $q =
p|_{d^\ell(\star)}$ for some $p\in \calb^\star(\Delta_n X)$ and
$\ell\in \mathbb{Z}_{\geq 1}$. Since $q|_s=p|_{d^\ell(s)}$ is normal, we have
$\lbar{s} \in \Delta_{n-\ell} X$ by Lemma \mref{lemma:diffnormal}.
Furthermore, $s_i< \lbar{s}$ implies that $s_i \in  \Delta_{n} X$.
Thus by
Definition \ref{defweakmonomial} and Proposition~\ref{weakmonomial}, we have $\lbar{q|_{s_i}} < q|_{\lbar{s}}$. So $\lbar{q|_s} = q|_{\lbar{s}}$.
\end{proof}

\begin{lemma}
Let $u,v\in \calb(\Delta_n X)$ with $u > v$ and $q\in
\calb^\star(\Delta_n X)$. If $q|_u$ is normal, then either $q|_v=0$ or
$q|_v$ is also normal. \mlabel{leqnormal}
\end{lemma}

\begin{proof}
Suppose that $q|_v$ is not normal. Then $q|_v \notin \calb(\Delta_n
X)$. We have the following cases to consider.

\noindent
{\bf Case I.}  $\dep(v)\geq 2$, that is, $v \in \cm(\Delta_n
X)P(\calb(\Delta_n X))$, and $q = p|_{\star P(w)}$ for some $p\in
\calb^\star(\Delta_n X)$ and $w\in \calb(\Delta_n X)$.
Since $u > v$, it follows that $\mathrm{dep}(u)\geq
\mathrm{dep}(v)\geq 2$ and so $u \in \cm(\Delta_n X)P(\calb(\Delta_n
X))$. This implies that $q|_u$ can be reduced by the Rota-Baxter
relation and so $q|_u \notin \calb(\Delta_n X)$. Hence $q|_u$ is not
normal, a contradiction.

\noindent
{\bf Case II.} $q = p|_{d^\ell(\star)}$ for some $p\in
\calb^\star(\Delta_n X)$ and $\ell\geq 1$.
If $\dep(v)\geq 2$, then since $u > v$, we have $\mathrm{dep}(u)\geq
2$ and so $u\in \cm(\Delta_n X)P(\calb(\Delta_n X))$. This implies
that $q|_u$ is not normal, a contradiction. If $\dep(v)=1$, then $v
\in \cm(\Delta_n X)$. If further $\mathrm{deg}_{\Delta_n X}(v)\geq 2$,
then since $u
> v$, we have either $\mathrm{dep}(u)\geq 2$, or $\mathrm{dep}(u)=1$
and $\mathrm{deg}_{\Delta_n X}(u)\geq 2$. In either case, we have
that $q|_u$ is not normal, a contradiction. Thus we must have
$\dep(v)=1$ and $\deg_{\Delta_n X}(v)=1$. So $v=x^{(r)}, r\geq 1.$
Since $q|_v=p|_{d^{\ell}(v)}=p|_{x^{(\ell+r)}}$ is supposed to be
not normal, we have $\ell+r>n$. That is, $q|_v=p|_{d^\ell(v)}=0.$
\end{proof}

\section{Composition-Diamond lemma}
\mlabel{sec:cd}
In this section, we shall establish the composition-diamond lemma for the
order $n$ free commutative differential Rota-Baxter algebra
$\sha(\bfk[\Delta_n X])$.
\begin{defn}
\begin{enumerate}
\item Let $u,w \in \calb(\Delta_n X)$. We call $u$ a {\bf
subword} of $w$ if $w$ is in the operated ideal of $\frakC(\Delta_n
X)$ generated by $u$. In terms of $\star$-words, $u$ is a subword of
$w$ if there is a $q\in \calb^\star(\Delta_n X)$ such that $w=q|_u$.
\item
Let $u_1$ and $u_2$ be two subwords of $w$. $u_1$ and $u_2$ are
called {\bf separated} if $u_1\in \cm(\Delta_n X)$, $u_2\in
\calb(\Delta_n X)$ and there is a $q\in \calb^{\star_1,
\star_2}(\Delta_n X)$ such that $w=q|_{u_1,u_2}$.
\item
For any $u\in \calb(\Delta_n X)$, $u$ can be expressed as $u =u_{1}
\cdots u_{k}$, where $u_1, \cdots, u_{k-1} \in \Delta_n X$ and
$u_{k} \in \Delta_n X\cup P(\calb(\Delta_n X))$. The integer $k$ is called
the {\bf breath} of $u$ and is denoted by $\mathrm{bre}(u)$.
\item
Let $f,g\in \calb(\Delta_n X)$. A pair $(u,v)$ with $u\in
\calb(\Delta_n X)$ and $v\in \cm(\Delta_n X)$ is called an {\bf
intersection pair} for $(f,g)$ if the differential
Rota-Baxter monomial $w:= fu$ equals $vg$ and satisfies
$\mathrm{bre}(w)<\mathrm{bre}(f) + \mathrm{bre}(g)$. Then we call
$f$ and $g$ to be {\bf overlapping}. Note that if $f$ and $g$ are
overlapping, then $f\in \cm(\Delta_n X)$.
\end{enumerate}
 \mlabel{lemma:overlape}
\end{defn}

There are four kinds of compositions.

\begin{defn}
{\rm
Let $\leq$ be a weakly monomial order on $\calb(\Delta_n X)$ and
$f,g\in \mathbf{k}\calb(\Delta_n X)$ monic with respect to $\leq$.

\begin{enumerate}
\item If $\overline{f}\in C( \Delta_n X) P(\calb(\Delta_n X))$, then define a {\bf composition of
(right) multiplication} to be $fu$ where $u\in C(\Delta_n X) P(\calb(\Delta_n X))$.
\item If $\overline{f} \notin \Delta_n X$,
then define a {\bf composition of derivation} to be $d^\ell(f)$, where
$\ell\in \mathbb{Z}_{\geq 1}$.
\item If there is an intersection pair $(u,v)$ for $(\overline{f},
\overline{g})$, then we define
$$(f,g)_w :=(f,g)^{u,v}_{w} := fu-vg$$ and call it an {\bf intersection
composition} of $f$ and $g$.
\item If there exists a $q\in \calb^\star(\Delta_n X)$ such that
$w:=\overline{f}=q|_{\overline{g}}$, then we define $(f,g)_w
:=(f,g)^{q}_{w} := f-q|_g$ and call it an {\bf including
composition} of $f$ and $g$ with respect to $q$. Note that if this
is the case, then $q|_g$ is normal.
\end{enumerate}
}
\end{defn}
In the last two cases, $(f,g)_w$ is called the {\bf ambiguity} of the composition.
\begin{defn}
{\rm Let $\leq$ be a weakly monomial order on $\calb(\Delta_n X)$,
$S\subseteq \mathbf{k}\calb(\Delta_n X)$ be a set of monic
differential Rota-Baxter polynomials and $w\in
\calb(\Delta_n X)$.

\begin{enumerate}
\item An element $h\in \bfk \frakB(\Delta_n X)$ is called {\bf trivial \text{ mod } $[S]$} and denote this by
    $$h\equiv 0 \text{ mod } [S]$$
if $h=\sum_i c_i q_i|_{s_i}, $
where $c_i\in \bfk$, $q_i\in \calb^\star(\Delta_n X)$, $s_i\in S$, $q_i|_{s_i}$ is normal
and $q_i|_{\overline{s_i}} \leq \overline{h}$.
This applies in particular to a composition of multiplication $fu$ and a composition of derivation $d^\ell(f)$ where $f\in \bfk\frakB(\Delta_n X), u\in P(\frakB(\Delta_nX))$ and $\ell\geq 1$.
We use $\text{ mod } [S]$ to distinguish from the usual notion of $u\equiv 0 \text{ mod } (S)$ when $u$ is in the ideal generated by $S$.
\item For $u,v\in \mathbf{k}\calb(\Delta_n X)$, we call $u$ and $v$
{\bf congruent modulo $(S,w)$} and denote this by $$u \equiv v \text
{ mod } (S,w)$$ if $u-v = \sum_{i}c_iq_i|_{s_i}$, where $c_i\in
\mathbf{k}$, $q_i\in \calb^\star(\Delta_n X)$, $s_i\in S$,
$q_i|_{s_i}$ is normal and $\overline{q_i|_{s_i}} < w$.
\item For $f,g\in \mathbf{k}\calb(\Delta_n X)$ and suitable $u,v$ or $q$
that give an intersection composition $(f,g)^{u,v}_{w}$ or an including composition $(f,g)^{q}_{w}$, the composition is called
{\bf trivial modulo $(S,w)$} if $$(f,g)^{u,v}_{w} \text{ or }
(f,g)^{q}_{w}\equiv 0 \text{ mod } (S,w).$$

\item The set $S\subseteq \mathbf{k}\calb(\Delta_n X)$ is a {\bf Gr\"{o}bner-Shirshov basis} if all
compositions of multiplication and derivation are trivial mod $[S]$, and, for $f, g\in S$, all
intersection compositions $(f,g)^{u,v}_{w}$ and all including
compositions $(f,g)^{q}_{w}$ are trivial modulo $(S,w)$.
\end{enumerate}
}
\end{defn}

We give some preparational lemmas before establishing the Composition-Diamond Lemma.

\begin{lemma} \label{separated}
Let $\geq$ be the weakly monomial order on $\calb(\Delta_n X)$ defined in Eq.~(\ref{eq7}), $s_1,s_2\in \mathbf{k}\calb(\Delta_n
X)$, $q_1,q_2 \in \calb^\star(\Delta_n X)$ and $w\in \calb(\Delta_n
X)$ such that $w= q_1|_{\overline{s_1}} = q_2|_{\overline{s_2}}$,
where each $q_i|_{s_i}$ is normal, $i=1,2$. If $\overline{s_1}$ and
$\overline{s_2}$ are separated in $w$, then $q_1|_{s_1} \equiv
q_2|_{s_2}$ mod $(S,w)$. \mlabel{lemma:separated}
\end{lemma}

\begin{proof} Let $q\in \calb^{\star_1,\star_2}(\Delta_n X)$ be the
$(\star_1,\star_2)$-DRB monomial
obtained by replacing this occurrence of $\overline{s_1}$ in $w$ by
$\star_1$ and this occurrence of $\overline{s_2}$ in $w$ by
$\star_2$. Then we have
$$q^{\star_1}|_{\overline{s_1}} = q_2, q^{\star_2}|_{\overline{s_2}}
= q_1 \text{ and } q|_{\overline{s_1}, \overline{s_2}} =
q_1|_{\overline{s_1}} = q_2|_{\overline{s_2}},$$ where in the first
two equalities, we have identified $\calb^{\star_2}(\Delta_n X)$ and
$\calb^{\star_1}(\Delta_n X)$ with $\calb^{\star}(\Delta_n X)$. Let
$s_1-\overline{s_1} = \sum_{i} c_iu_i$ and $s_2-\overline{s_2} =
\sum_{j} d_jv_j$ with $c_i,d_j\in \mathbf{k}$ and $u_i,v_j\in
\calb(\Delta_n X)$. Then by the linearity of $s_1$ and $s_2$ in
$q|_{s_1,s_2}$, we have
\begin{align*}
q_1|_{s_1} - q_2|_{s_2} &= (q^{\star_2}|_{\overline{s_2}})|_{s_1} -
(q^{\star_1}|_{\overline{s_1}})|_{s_2}\\
&= q|_{s_1,\overline{s_2}} - q|_{\overline{s_1},s_2}\\
&=q|_{s_1,\overline{s_2}} - q|_{s_1,s_2} + q|_{s_1,s_2} -
q|_{\overline{s_1},s_2}\\
&= -q|_{s_1, s_2-\overline{s_2}} + q|_{s_1-\overline{s_1},s_2}\\
&= -(q^{\star_2}|_{s_2-\overline{s_2}})|_{s_1} +
(q^{\star_1}|_{s_1-\overline{s_1}})|_{s_2}\\
&= -\sum_{j} d_j(q^{\star_2} | _{v_j})|_{s_1} + \sum_{i}
c_i(q^{\star_1} | _{u_i})|_{s_2}\\
&= -\sum_{j} d_j q|_{s_1,v_j} + \sum_{i} c_i q|_{u_i, s_2}.
\end{align*}
Since $(q^{\star_1}|_{s_1})|_{\lbar{s_2}} = q|_{s_1, \lbar{s_2}} =
(q^{\star_2}|_{\lbar{s_2}})|_{s_1} = q_1|_{s_1}$ is normal and $v_j <
\lbar{s_2}$, by Definition~\mref{defweakmonomial} and Proposition~\mref{weakmonomial}, we have
$$\lbar{q|_{s_1, v_j}} = \lbar{(q^{\star_1}|_{s_1})|_{v_j}} < \lbar{(q^{\star_1}|_{s_1})|_{\lbar{s_2}}}
 = \lbar{q_1|_{s_1}} = q_1|_{\lbar{s_1}} =w.$$
Similarly, since $(q^{\star_2}|_{s_2})|_{\lbar{s_1}} =
q|_{\lbar{s_1}, s_2}=(q^{\star_1}|_{\lbar{s_1}})|_{s_2} =
q_2|_{s_2}$ is normal and $u_i< \lbar{s_1}$, we have
$$\lbar{q|_{u_i, s_2}} = \lbar{(q^{\star_2}|_{s_2})|_{u_i}} < \lbar{(q^{\star_2}|_{s_2})|_{\lbar{s_1}}}
 = \lbar{q_2|_{s_2}} = q_2|_{\lbar{s_2}} =w.$$
Hence $q_1|_{s_1} \equiv q_2|_{s_2}$ mod $(S,w)$.
\end{proof}

For $q\in \calb^{\star}(\Delta_n X)$, let $\mathrm{dep}_{\star}(q)$ be the depth of the symbol $\star$ in $q$. For example, $\mathrm{dep}_{\star}(q) = 1$ if $q=P(\star)$ and
 $\mathrm{dep}_{\star}(q) = 2$ if $q=P(xP(\star))$.

\begin{lemma} \label{normal expression}
Let $\leq_n$ be the weakly monomial order on $\calb(\Delta_n X)$ defined in Eq.~(\ref{eq7}) and let $S\subseteq \mathbf{k}\calb(\Delta_n
X)$. If each composition of multiplication and derivation of
$S$ is trivial mod $[S]$, then, for $s\in S$ and $q\in
\calb^\star(\Delta_n X)$, $q|_s$ is trivial mod $[S]$:
$$q|_s = \sum_i c_i q_i|_{s_i},$$
where, for each $i$,  $0\neq c_i\in \bfk$, $s_i\in S$, $q_i|_{s_i}$
is normal and $q_i|_{\overline{s_i}} \leq \overline{q|_s}$.
\end{lemma}

\begin{proof}
We have the following two cases to consider.
\smallskip

\noindent
{\bf Case I.}  $q\in \calb^{\star}_{II}(\Delta_n X)$. In this case, we prove the result by induction on $\mathrm{dep}_{\star}(q)$.
If  $\mathrm{dep}_{\star}(q) = 0$, then $q = u\star v$, where $u \in \cm(\Delta_n X)$ and $v\in
\calb(\Delta_n X)$. If $s\in S$ is such that $\lbar{s}\in \cm(\Delta_n X)$  or $v\in \cm(\Delta_n
X)$, then it is obvious that $q|_s$ is normal by Definition \mref{normaldef} (c). Suppose $\lbar{s}, v\notin
\cm(\Delta_n X)$. Then $\lbar{s}, v\in \cm(\Delta_n X)P(\calb(\Delta_n X))$. Since the composition of multiplication of $S$ is trivial mod $[S]$, we have $$sv = \sum_i
d_i p_i|_{t_i},$$  where $0\neq d_i\in \bfk$, $t_i\in S$,
$p_i|_{t_i}$ is normal and $p_i|_{\overline{t_i}} \leq
\overline{sv}$. Let $q_i:= u p_i \in \calb^\star(\Delta_n
X)$. Since $u\in \cm(\Delta_n X)$, we have $q_i|_{t_i} =
u p_i|_{t_i}$ is normal. Moreover,
$$q|_s = usv = \sum_i d_i up_i|_{t_i} = \sum_i d_i
q_i|_{t_i},$$ where $q_i|_{\overline{t_i}} =
up_i|_{\overline{t_i}} \leq  u\overline{sv} =\lbar{
u sv } =  \overline{q|_{s}}$. Hence $q|_s$ is trivial mod $[S]$.

Suppose the claim has been proved for $q\in \cm(\Delta_nX)$ with $\dep_\star(q)=k\geq 0$ and consider $q$ with $\dep_\star(q)=k+1$. Then $q = uP(p)$, where $u \in \cm(\Delta_n X)$ and $p \in
\calb^\star_{II}(\Delta_n X)$ with $\dep_\star(p)=k$. By the induction hypothesis we
have $p |_s = \sum_i c_i p_i|_{t_i},$ where $0\neq c_i\in \bfk$,
$t_i\in S$, $p_i|_{t_i}$ is normal and $p_i|_{\overline{t_i}} \leq
\overline{p|_s}$. Let $q_i := uP(p_i) \in \calb^\star(\Delta_n X)$.
Then $q|_s = \sum_i c_iq_i|_{t_i}$, $q_i|_{t_i} = uP(p_i|_{t_i})$ is
normal and
$q_i|_{\lbar{t_i}} = uP(p_i|_{\lbar{t_i}}) \leq
uP(\lbar{p|_{s}}) = \lbar{uP(p|_{s})} = \lbar{q|_s},$ as desired. This completes the induction.

\noindent
{\bf Case II.} $q\in \calb^{\star}_{I}(\Delta_n X)$. Then $q =
p|_{d^{\ell}(\star)}$ for some $p\in \calb^\star(\Delta_n X)$ and
$\ell{\geq 1}$. Choose such an $\ell$ to be maximal so that
$p$ is in $ \calb^{\star}_{II}(\Delta_n X)$. By our hypothesis,
the composition of derivation is trivial mod $[S]$. So $d^\ell(s) =
\sum c_i p_i|_{s_i}$, where $0\neq c_i\in \bfk$, $s_i\in S$,
$p_i|_{s_i}$ is normal and $p_i|_{\overline{s_i}} \leq
\overline{d^\ell(s)}$. Since $p$ is in $ \calb^{\star}_{II}(\Delta_n
X)$, by Cases I that has been proved above, the result holds.
\end{proof}

\begin{lemma} \label{basis}
Let $\geq$ be the weakly monomial order on $\calb(\Delta_n X)$
defined in Eq.~(\ref{eq7}) and let $S\subseteq \mathbf{k}\calb(\Delta_n
X)$. If $S$ is a Gr\"{o}bner-Shirshov basis, then for each pair
$s_1,s_2\in S$ for which there exist $q_1,q_2\in
\calb^\star(\Delta_n X)$ and $w\in \calb(\Delta_n X)$ such that
$w=q_1|_{\overline{s_1}} = q_2|_{\overline{s_2}}$ with
$q_1|_{s_1}$ and $q_2|_{s_2}$ normal, we have $q_1 | _{s_1} \equiv q_2 |
_{s_2}$ mod $(S,w)$.

\mlabel{lemma:basis}
\end{lemma}

\begin{proof}  Let $s_1,s_2\in S$, $q_1,q_2\in
\calb^\star(\Delta_n X)$ and $w\in \calb(\Delta_n X)$ be such that
$w=q_1|_{\overline{s_1}} = q_2|_{\overline{s_2}}$. According to the
relative location of $s_1$ and $s_2$ in $w$, we have the following
three cases to consider.
\smallskip

\noindent
{\bf Case I.} $\overline{s_1}$ and $\overline{s_2}$ are separated in
$w$. This case is covered by Lemma \ref{separated}.
\smallskip

\noindent
{\bf Case II.} $\overline{s_1}$ and $\overline{s_2}$ are overlapping
in $w$. Then there are $u\in \calb(\Delta_n X)$ and $v\in \cm(\Delta
X)$ such that $w_1:=\overline{s_1}u=v\overline{s_2}$ is a subword in
$w$ with $\mathrm{bre}(w_1)< \mathrm{bre}(\overline{s_1})+
\mathrm{bre}(\overline{s_2})$. Since $S$ is a Gr\"{o}bner-Shirshov
basis, we have $$s_1u-vs_2 = \sum_j c_j p_j |_{t_j},$$ where $p_j |_
{t_j}$ is normal and $\overline{p_j |_{t_j}} = p_j
|_{\overline{t_j}} < \overline{s_1}u = v\overline{s_2} = w_1$.

Let $q\in \calb^{\star_1,\star_2}(\Delta_n X)$ be obtained from
$q_1$ by replacing $\star$ by $\star_1$, and the $u$ on the right of
$\star$ by $\star_2$. Let $p\in \calb^\star(\Delta_n X)$ be obtained
from $q$ by replacing $\star_1\star_2$ by $\star$. Then we have
$$q^{\star_2} |_u = q_1, q^{\star_1} |_v = q_2 \text{ and } p|_{\overline{s_1}u} =
q|_{\overline{s_1}, u} = q_1 |_{\overline{s_1}} =w,$$ where in the
first two equalities, we have identified $\calb^{\star_2}(\Delta_n
X)$ and $\calb^{\star_1}(\Delta_n X)$ with $\calb^{\star}(\Delta
X)$. Thus we have
$$q_1 |_{s_1} - q_2 |_{s_2} = (q^{\star_2}|_u)|_{s_1}
-(q^{\star_1}|_v)|_{s_2} = p|_{s_1u-vs_2} = \sum_j c_j p|_{p_j|_{t_j}}.
$$
Since $\overline{p_j |_{t_j}} < w_1$ and $p|_{w_1} =w \in
\calb(\Delta_n X)$ is normal, we have $p|_{\overline{p_j
|_{t_j}}}$ is either zero or normal by Lemma
\ref{leqnormal}. If $p|_{\overline{p_j |_{t_j}}}=0$, there is
nothing to prove. If $p|_{\overline{p_j |_{t_j}}}$ is normal, then
by Lemma \ref{normalequiv}, we have
$\overline{p|_{p_j|_{t_j}}} = p|_{\overline{p_j |_{t_j}}} <
p|_{w_1} = w.$ Hence
$q_1 |
_{s_1} \equiv q_2 | _{s_2} \text{ mod } (S, w).$
\smallskip

\noindent
{\bf Case III.} One of $\overline{s_1}$ or $\overline{s_2}$ is a
subword of the other. Without loss of generality, we may suppose
that $\overline{s_1}=q|_{\overline{s_2}}$ for some $q\in
\calb^\star(\Delta_n X)$. Since $\overline{s_1}=q|_{\overline{s_2}}
\in \calb(\Delta_n X)$, it follows that $q|_{s_2}$ is normal  by
Definition \ref{normaldef} and $\overline{q|_{s_2}} =
q|_{\overline{s_2}}$. For the  inclusion composition
$(s_1,s_2)^q_{\overline{s_1}}$, since $S$ is a Gr\"{o}bner-Shirshov
basis, we have
$(s_1,s_2)^q_{\overline{s_1}} = s_1-q|_{s_2} = \sum_j c_j p_j |_{t_j},$
where $c_j\in \mathbf{k}$, $p_j\in \calb^\star(\Delta_n X)$, $t_j\in
S$ and $p_j|_{t_j}$ is normal with $\overline{p_j|_{t_j}} <
\overline{s_1}$. Let $p\in \calb^\star(\Delta_n X)$ be obtained from
$q_1$ by replacing $\star$ with $q$. Then
$w= q_2|_{\overline{s_2}} = q_1|_{\overline{s_1}} = q_1|_{q|_{\overline{s_2}}} = p|_{\overline{s_2}}.$
Since $S$ is a Gr\"{o}bner-Shirshov basis, by Cases I and II, we
have
$$p|_{\overline{s_2}} - q_2|_{\overline{s_2}} = \sum_i d_i r_i|_{v_i},$$
where $d_i\in \mathbf{k}$, $r_i\in \calb^\star(\Delta_n X)$, $v_i\in
S$ and $r_i|_{v_i}$ is normal with $\overline{r_i|_{v_i}} =
r_i|_{\overline{v_i}} < q_2|_{\overline{s_2}}=w$. So
\begin{align*}
q_2|_{s_2} - q_1|_{s_1} &= (p|_{s_2} - \sum_i d_i r_i|_{v_i}) -
q_1|_{s_1}\\
&=p|_{s_2} -q_1|_{s_1} -\sum_i d_i r_i|_{v_i}\\
&= q_1|_{q|_{s_2}} -q_1|_{s_1} -\sum_i d_i r_i|_{v_i}\\
&= -q_1|_{s_1-q|_{s_2}} -\sum_i d_i r_i|_{v_i}\\
&= -\sum_j c_j q_1|_{p_j|_{t_j}}-\sum_i d_i r_i|_{v_i}.
\end{align*}
Since $\overline{p_j |_{t_j}} < \lbar{s_1}$ and $q_1|_{\lbar{s_1}}
=w \in \calb(\Delta_n X)$ is normal by our hypothesis,
we have $q_1|_{\overline{p_j |_{t_j}}}=0$ or $q_1|_{\overline{p_j
|_{t_j}}}$ is normal by Lemma \ref{leqnormal}. If
$q_1|_{\overline{p_j |_{t_j}}}=0$, there is noting to prove. If
$q_1|_{\overline{p_j |_{t_j}}}$ is normal, then by Lemma~\ref{normalequiv},
$\overline{q_1|_{p_j|_{t_j}}}
=q_1|_{\overline{p_j|_{t_j}}} < q_1|_{\overline{s_1}} =w.$ Hence
$q_2|_{s_2} - q_1|_{s_1} \equiv 0 \text{ mod } (S, w).$
\end{proof}

\begin{lemma}\label{expression}\mlabel{lemma:expression}
Let $\leq_n$ be the weakly monomial order on $\calb(\Delta_n X)$
defined in Eq.~(\ref{eq7}), $S \subseteq \mathbf{k}\calb(\Delta_n
X)$ and $\mathrm{Irr(S)}:= \calb(\Delta_n X) \setminus \{
q|_{\overline{s}} \mid q\in \calb^\star(\Delta_n X), s\in S ,
q|_{s}\text{ is normal } \}$. Then for any $f\in
\mathbf{k}\calb(\Delta_n X)$, $f$ has an expression
$$f = \sum_i c_i u_i +  \sum_j d_j q_j|_{s_j}, $$
where $0\neq c_i,d_j \in \bfk, u_i\in \mathrm{Irr(S)},
\overline{u_i} \leq\lbar{f}$, $s_j\in S$, $q_j|_{s_j}$ is normal and
$q_j|_{\lbar{s_j}} \leq \lbar{f}$.
\end{lemma}

\begin{proof}
Suppose the lemma does not hold and let $f$ be a counterexample with minimal $\lbar{f}$. Write
$f = \sum_i c_i u_i$ where $0\neq c_i\in \bfk$, $u_i\in
\calb(\Delta_n X)$ and $u_1>u_2>\cdots$. If $u_1\in
\mathrm{Irr(S)}$, then let $f_1 := f-c_1u_1$. If $u_1\notin
\mathrm{Irr}(S)$, that is, there exists $s_1\in S$ such that $u_1 =
q_1|_{s_1}$ and $q_1|_{s_1}$ is normal, then let $f_1 := f -
c_1q_1|_{s_1}$. In both cases $\lbar{f_1} < \lbar{f}$. By the
minimality of $f$, we have that $f_1$ has the desired
expression. Then $f$ also has the desired expression. This is a
contradiction.
\end{proof}

Now we are ready to derive the Composition-Diamond Lemma.

\begin{theorem}(Composition-Diamond Lemma)\label{CD lemma} Let $\geq$ be the weakly monomial order on $\calb(\Delta_n X)$
defined in Eq.~(\ref{eq7}), $S_n$ a set of monic DRB polynomials in $\bfk \calb(\Delta_n
X)$ and $\mathrm{Id(S_n)}$ the Rota-Baxter ideal of
$\mathbf{k}\calb(\Delta_n X)$ generated by $S_n$. Then the following
conditions are equivalent:
\begin{enumerate}
\item $S_n$ is a Gr\"{o}bner-Shirshov basis in
$\mathbf{k}\calb(\Delta_n X)$.
\mlabel{it:cda}
\item If $0\neq f\in \mathrm{Id(S_n)}$, then $\overline{f} =
q|_{\overline{s}}$ for some $q\in \calb^\star(\Delta_n X)$, $s\in S_n$
and $q|_s$ is normal.
\mlabel{it:cdb}
\item
$\mathrm{Irr(S_n)}:= \calb(\Delta_n X) \setminus
\{ q|_{\overline{s}} \mid q\in \calb^\star(\Delta_n X), s\in S_n ,
q|_{s}\text{ is normal} \}$ is a $\mathbf{k}$-basis of
$\mathbf{k}\calb(\Delta_n X)/\mathrm{Id(S_n)}$.
In other words, $\mathbf{k} \mathrm{Irr(S_n)} \oplus \mathrm{Id(S_n)} = \mathbf{k}
\calb(\Delta_n X)$.
\mlabel{it:cdc}
\end{enumerate} \mlabel{thm:CD lemma}
\end{theorem}

\begin{proof} \mref{it:cda} $\Rightarrow$ \mref{it:cdb}: Let $0\neq f\in \mathrm{Id(S_n)}$.
Then by Lemmas \ref{operator ideal} and \ref{normal expression},
\begin{equation}\label{ideal}
f = \sum_{i=1}^{k} c_iq_i|_{s_i}, \text{ where } 0\neq c_i \in
\mathbf{k},  s_i\in S_n, q_i|_{s_i} \text{ is normal}, 1\leq i\leq k.
\end{equation}
Let $w_i = q_i |_{s_i}, 1\leq i\leq k$. We rearrange them in non-increasing order
by
$$w_1=w_2=\cdots=w_m>w_{m+1}\geq \cdots \geq w_k.$$
If for each $0\neq f\in \mathrm{Id(S_n)}$, there is a choice of the
above sum such that $m=1$, then $\overline{f} = q_1|_{s_1}$ and we
are done. So suppose the implication (a) $\Rightarrow({\rm b})$ does
not hold. Then there is a $0\neq f\in \mathrm{Id(S_n)}$ such that for
any expression in Eq.~(\ref{ideal}), we have that $m\geq 2$. Fix
such an $f$ and choose an expression in Eq.~(\ref{ideal}) such that
$q_1|_{\overline{s_1}}$ is minimal and then with $m\geq 2$ minimal, that is, with
the fewest $q_i|_{s_i}$ such that
$q_i|_{\overline{s_i}}=q_1|_{\overline{s_1}}$. Since $m\geq 2$, we
have $q_1|_{\overline{s_1}}=w_1=w_2=q_2|_{\overline{s_2}}$.

Since $S_n$ is a Gr\"{o}bner-Shirshov basis in $\mathbf{k}
\calb(\Delta_n X)$, by Lemma \ref{basis}, we have
$$q_2|_{\overline{s_2}} - q_1|_{\overline{s_1}} =\sum_j d_j p_j|_{r_j},$$
where $d_j\in \mathbf{k}$, $r_j\in S_n$, $p_j\in \calb^\star(\Delta
X)$ and $p_j|_{r_j}$ are normal with $p_j|_{\overline{r_j}} < w_1$.
Hence
$$
f = \sum_{i=1}^{k} c_iq_i|_{s_i} = (c_1+c_2)q_1|_{s_1} + c_3 q_3|_{s_3} + \cdots + c_m q_m|_{s_m} +
\sum_{i=m+1}^{k} c_iq_i|_{s_i} + \sum_j c_2 d_j p_j|_{r_j}.
$$
By the minimality of $m$, we must have $c_1+c_2 = c_3 =\cdots = c_m
= 0$. Then we obtain an expression of $f$ in the form of
Eq.~(\ref{ideal}) for which $q_1|_{\overline{s_1}}$ is even smaller,
a contradiction.

\mref{it:cdb} $\Rightarrow \mref{it:cdc}$: Obviously $0\in \mathbf{k} \mathrm{Irr(S_n)}
+\mathrm{Id(S_n)} \subseteq \mathbf{k} \calb(\Delta_n X)$. Suppose the
inclusion is proper. Then $\mathbf{k} \calb(\Delta_n X) \setminus
(\mathbf{k} \mathrm{Irr(S_n)} + \mathrm{Id(S_n)})$ contains only nonzero
elements. Let $f\in \mathbf{k} \calb(\Delta_n X) \setminus
(\mathbf{k} \mathrm{Irr(S_n)} + \mathrm{Id(S_n)}) $ be such that

$$\overline{f} = \mathrm{min} \{\overline{g} \mid g\in \mathbf{k}
\calb(\Delta_n X) \setminus (\mathbf{k} \mathrm{Irr(S_n)} +
\mathrm{Id(S_n)}) \}.$$

\noindent
{\bf Case I.} $\overline{f} \in \mathrm{Irr(S_n)}$. Then $f\neq
\overline{f}$ since $f\notin \mathrm{Irr(S_n)}$. By
$\overline{f-\overline{f}} < \overline{f}$ and the minimality of
$\overline{f}$, we must have $f-\overline{f} \in \mathbf{k}
\mathrm{Irr(S_n)} + \mathrm{Id(S_n)}$ and so $f\in \mathbf{k}
\mathrm{Irr(S_n)} + \mathrm{Id(S_n)}$, a contradiction.
\smallskip

\noindent
{\bf Case II.} $\overline{f} \notin \mathrm{Irr(S_n)}$. Then by the
definition of $\mathrm{Irr(S_n)}$, we have $\overline{f} =
q|_{\overline{s}}$ for some $q\in \calb^\star(\Delta X)$, $s\in S_n$
and $q|_s$ is normal. Thus $\lbar{q|_s} = q|_{\lbar{s}} = \lbar{f}$
and so $\overline{f-q|_s} < \overline{f}$. If $f=q|_s$, then $f\in
\mathrm{Id(S_n)}$, a contradiction. If $f\neq q|_s$, then $f-q|_s \neq
0$ with $\overline{f-q|_s} < \overline{f}$. By the minimality of
$\overline{f}$, we have $f-q|_s \in \mathbf{k} \mathrm{Irr(S_n)} +
\mathrm{Id(S_n)}$. This implies that $f\in \mathbf{k} \mathrm{Irr(S_n)}
+ \mathrm{Id(S_n)}$, again a contradiction.

Hence $\mathbf{k} \mathrm{Irr(S_n)} + \mathrm{Id(S_n)} = \mathbf{k}
\calb(\Delta_n X)$. Suppose $\mathbf{k} \mathrm{Irr(S_n)} \cap
\mathrm{Id(S_n)} \neq 0$ and let $0\neq f\in \mathbf{k}
\mathrm{Irr(S_n)} \cap \mathrm{Id(S_n)}$. Then
$$f = c_1v_1 +c_2v_2 + \cdots + c_kv_k,$$
where $v_1 >v_2>\cdots>v_k \in \mathrm{\mathrm{Irr(S_n)}}$. Since
$f\in \mathrm{Id(S_n)}$, by Item (b), we have $v_1 = \overline{f} =
q|_{\overline{s}}$ for some $q\in \calb^\star(\Delta_n X)$, $s\in S_n$
and $q|_s$ is normal. This is a contradiction to the construction of
$\mathrm{\mathrm{Irr(S_n)}}$. Therefore $\mathbf{k} \mathrm{Irr(S_n)}
\oplus \mathrm{Id(S_n)} = \mathbf{k} \calb(\Delta_n X)$ and
$\mathrm{Irr(S_n)}$ is a $\mathbf{k}$-basis of $\mathbf{k}\calb(\Delta
X)/\mathrm{Id(S_n)}$.

\mref{it:cdc} $\Rightarrow \mref{it:cda}:$ Suppose $f,g\in S_n$ give an intersection or
including composition. Let $F=fu$ and $G=vg$ in the case of
intersection composition and let $F=f$ and $G=q|_g$ in the case of
including composition. Then we have $w:= \overline{F} =\overline{G}
$. If $(f,g)_w = F-G=0$, then there is nothing to prove. If $(f,g)_w
\neq 0$, then we have
$$(f,g)_w = \sum_i c_iq_i,\quad 0\neq c_i\in \mathbf{k}, q_1>q_2>\cdots>q_k \in \calb(\Delta_n X).$$
Then $q_i <\overline{F} =\lbar{G}= w$. Since $(f,g)_w \in
\mathrm{Id(S_n)}$, by Item(c), we have that the $q_i$ are not in
$\mathrm{Irr(S_n)}$. By the definition of $\mathrm{Irr(S_n)}$, there are
$q_i\in \calb^\star(\Delta_n X)$, $s_i\in S_n$ such that $q_i=
q_i|_{\lbar{s_i}}$ and $q_i|_{s_i}$ is normal. Since
$\lbar{q_i|_{s_i}} = q_i|_{\lbar{s_i}} < w$, we have $(f,g)_w \equiv
0$ mod $(S_n,w)$.

For any composition of multiplication $fu$ where $f\in S_n$
and $u \in C(\Delta_n X) \calb(\Delta_n X)$, we have $fu \in \Id(S_n)$. By
Lemma \ref{expression}, it follows that $fu = \sum_i c_i
q_i|_{s_i}$ where $0\neq c_i\in \bfk$, $s_i\in S_n$, $q_i\in
\calb^\star(\Delta_n X)$, $q_i|_{s_i}$ is normal and
$q_i|_{\overline{s_i}} \leq \overline{fu}$. Hence the composition
of multiplication is trivial mod $[S_n]$.

For any composition of derivation $d^\ell(f)$ where $f\in S_n$ and
$\ell\in \mathbb{Z}_{\geq 1}$, we have $d^\ell(f) \in
\mathrm{Id(S_n)}$. By Lemma \ref{expression}, we have $d^\ell(f)
= \sum_i c_i q_i|_{s_i}$ where $0\neq c_i\in \bfk$, $s_i\in S_n$,
$q_i\in \calb^\star(\Delta_n X)$, $q_i|_{s_i}$ is normal and
$q_i|_{\overline{s_i}} \leq \overline{fP(v)}$. Hence the composition
of derivation $d^\ell(f)$ is trivial mod $[S_n]$.

Therefore $S_n$ is a Gr\"obner-Shirshov basis.
\end{proof}

\section{Gro\"{o}bner-Shirshov bases and free commutative integro-differential algebras}
\mlabel{sec:gs}

In this section we begin with a finite set $X$ and prove that the relation ideal of the free commutative differential Rota-Baxter algebra on $X$ of order $n$, where $n\geq 1$, that defines the corresponding commutative integro-differential algebra of order $n$ possesses a Gr\"obner-Shirshov basis. This is done in Section~\mref{ss:gsb}. Then in Section~\mref{ss:bases}, we apply the Composition-Diamond Lemma in Theorem~\mref{thm:CD lemma} to construct a canonical basis for the commutative integro-differential algebra of order $n$.
Taking $n$ to go to the infinity, we obtain a canonical basis of the free commutative integro-differential algebra on the finite set $X$. Finally for any well-ordered set $X$, by showing that the canonical basis of the free commutative integro-differential algebra on each finite subset of $X$ is compatible with the inclusion of the subset in $X$, we obtain a canonical basis of the free commutative integro-differential algebra on $X$.

\subsection{Gr\"obner-Shirshov basis}
\mlabel{ss:gsb}
We begin with a lemma that simplifies the defining ideal of the integro-differential algebra.

\begin{lemma}
Let $X$ be a finite set and let $\sha(\bfk[\Delta_n X])$ be the free commutative differential Rota-Baxter algebra on $X$. The differential Rota-Baxter ideal of $\sha(\bfk[\Delta_n X])$ generated by the set
$$\left\{
    P(d(u) P(v))- uP(v)+ P(uv) + \lambda P(d(u) v)\,\big|\,  u, v\in \sha(\bfk[\Delta_n X]).
\right\}
$$
is generated by
\begin{equation}
 S_n:= \left\{P(d(u) P(v))- uP(v)+ P(uv) + \lambda P(d(u) v)\,\big|\, u, v\in  \sha(\bfk[\Delta_n X]), u\notin P( \sha(\bfk[\Delta_n X]))\right\}.  \label{eq2}
\end{equation}
\end{lemma}

\begin{proof}
If $u$ is in $P(\sha(\bfk[\Delta_n X]))$, let $u=P(\hat{u})$
for some $\hat{u}\in \sha(\bfk[\Delta_n X])$. Then $P(d(u) P(v))- uP(v)+ P(uv) + \lambda P(d(u) v)$ vanishes since $P$ is a Rota-Baxter algebra. This proves the lemma.
\end{proof}

We show that $S_n$ is a Gr\"{o}bner-Shirshov basis of the ideal $\Id(S_n)\subseteq \sha(\bfk[\Delta_n X])$.

\begin{lemma}
Let $\phi(u,v) \in S_n$ with $u\in  \calb(\Delta_n X) \setminus P(
\calb(\Delta_n X))$ and $v\in \calb(\Delta_n X)$. Then
$\lbar{\phi(u,v)} = 1\otimes \lbar{d(u_0)} \otimes w$ for some
$u_0\in \cm(\Delta_n X)$ and $w \in \calb(\Delta_n X)$.
\mlabel{lemma:leading term}
\end{lemma}

\begin{proof} Let $u = u_0 \otimes \hat{u}$ with $1\neq
u_0 \in \cm(\Delta_n X) $ and $\hat{u}\in \calb(\Delta_n X)$
(take $\hat{u}=1\in \bfk$ when $u\in \cm(\Delta_n X)$). Then
\begin{equation}
\begin{aligned}
\lbar{\phi(u,v)} &= \lbar{P(d(u)P(v))} = \lbar{P(d(u_0 \otimes
\hat{u}) (1\otimes v))} = \lbar{P( d(u_0) \otimes (\hat{u} \sha_\lambda v) )}\\
&= \lbar{ P(d(u_0) \otimes w) }=P(\lbar{d(u_0)} \otimes w) = 1
\otimes \lbar{d(u_0)} \otimes w,
\end{aligned}
\end{equation}
where $w = \lbar{\hat{u} \sha_\lambda v} \in \calb(\Delta_n X)$.
\end{proof}

By the above lemma, we see that $\lbar{\phi(u,v)} \in
P(\calb(\Delta_n X))$ and so $\lbar{\phi(u,v)} \notin \cm(\Delta_n
X)$. So from Definition \mref{lemma:overlape}, there is no
intersection compositions in $S_n$. The following two lemmas show
that other kinds of compositions in $S_n$ are trivial.

\begin{lemma}
The compositions of multiplication and derivation are trivial
mod $[S_n]$. \mlabel{lemma:comtrivial1}
\end{lemma}

\begin{proof}
 Let $$f := \phi(u,v) := P(d(u) P(v))- uP(v)+ P(uv) + \lambda
P(d(u) v) \in S_n,$$ where $u\in  \calb(\Delta_n X) \setminus P(
\calb(\Delta_n X))$ and $v\in \calb(\Delta_n X)$. First, we check
that the compositions of derivation are trivial mod $(S_n)$. By Eq.~(\ref{eq:diffl}) and Eq.~(\ref{eq:fft}), we have

$$d(f) = d(u)P(v) - d(uP(v)) +uv + \lambda d(u)v =d(u)P(v)  - d(u)P(v) - uv -\lambda d(u)v +uv + \lambda d(u)v=0.$$
Hence $d^\ell(f) \equiv 0$ mod $[S_n]$ for any $\ell\in{\geq 1}$.

Next, we check that the compositions of multiplication
$\phi(u,v)w_0 P(w)$ with $w_0 \in C(\Delta_n X)$ and $w\in \calb(\Delta_n X)$ are trivial. Since $w_0 \in C(\Delta_n X)$, it is sufficient to show that $\phi(u,v)P(w)$ is trivial. Note that
$\overline{\phi(u,v)}\in P( \calb(\Delta_n X) )$ by Lemma
\ref{lemma:leading term}. From Eq.~(\ref{eq:rb}) we obtain
\begin{equation}
\begin{aligned} \label{eq5}
\phi(u,v) P(w) =& P(d(u)P(v))P(w) -(uP(v)) P(w) + P(uv)P(w) + \lambda  P(d(u) v)P(w)\\
=& P(P(d(u)P(v))w) + P(d(u)P(v)P(w)) + \lambda P(d(u)P(v)w) \\
&- u P(v)P(w) + P(uv)P(w) + \lambda  P(d(u) v) P(w) \\
=& P(P(d(u)P(v))w) + P(d(u)P(P(v)w + vP(w) + \lambda vw)) + \lambda
P(d(u)P(v)w)\\
&- u P(v)P(w) + P(uv)P(w) + \lambda  P(d(u) v)P(w)
\end{aligned}
\end{equation}
Since $ \phi(u,v) = P(d(u) P(v))- uP(v)+ P(uv) + \lambda P(d(u) v)$,
we have

\begin{align}\label{eq3}
P(P(d(u)P(v))w) =P(\phi(u,v)w) + P(uP(v)w) - P(P(uv)w) -\lambda
P(P(d(u)v)w),
\end{align}
\begin{equation}
\begin{aligned} \label{eq4}
&P(d(u)P(P(v)w + vP(w) + \lambda vw))\\
=& \phi(u,P(v)w + vP(w) + \lambda vw ) + uP(P(v)w + vP(w) + \lambda vw) \\
& - P(u(P(v)w + vP(w) + \lambda vw)) - \lambda P(d(u) (P(v)w + vP(w) + \lambda vw)) \\
=& \phi(u,P(v)w + vP(w) + \lambda vw) + uP(wP(v)) + uP(vP(w)) +
\lambda uP(vw)
-P(uwP(v))\\
&  - P(uvP(w)) - \lambda P(uvw) -\lambda P(d(u) wP(v)) -\lambda
P(d(u)vP(w)) -\lambda^2 P(d(u)vw)
\end{aligned}
\end{equation}
and
\begin{equation}
\begin{aligned} \label{eq6}
&- u P(v)P(w) + P(uv)P(w) + \lambda  P(d(u) v)P(w)\\
=& -uP(P(v)w) -uP(vP(w)) -\lambda u P(vw) + P(P(uv)w)+ P(uvP(w)) +
\lambda P(uvw) \\
&+ \lambda P(P(d(u)v)w) +\lambda P(d(u)vP(w)) +\lambda^2 P(d(u)vw).
\end{aligned}
\end{equation}
Substituting Eq.~(\ref{eq3}), Eq.~(\ref{eq4}) and Eq.~(\ref{eq6})
into Eq.~(\ref{eq5}), we have
\begin{align*}
\phi(u,v)P(w) = P(\phi(u,v)w) + \phi(u,wP(v))  + \phi(u, vP(w)) +
\lambda \phi(u,  vw)
\end{align*}
The last three terms are
already in $S_n$ and hence are of the form $q|_s$ with $q=\star$ and $s\in S_n$. So we just need to bound the leading terms. Note that
$$\lbar{P(aP(b))}, \lbar{P(bP(a))},
\lbar{P(ab)} \leq \lbar{P(a)P(b)}  \text{ for } a,b\in
\calb(\Delta_n X).$$ So we have
$$\overline{\phi(u,wP(v) ) } =  \lbar{P(d(u)P(wP(v)))} \leq \lbar{ P(d(u) P(v)P(w))}
\leq \lbar{P(d(u)P(v))P(w)}= \overline{\phi(u,v)P(w)}.$$
We similarly show that $\lbar{\phi(u, vP(w))}, \lbar{\phi(u,vw)} \leq
\lbar{\phi(u,v)P(w)}$. So $\phi(u,wP(v))  + \phi(u, vP(w)) + \lambda
\phi(u,  vw) \equiv 0$ mod $[S_n]$. Hence $\phi(u,v) P(w) \equiv 0$
mod $[S_n]$ if and only if $ P(\phi(u,v)w) \equiv 0$ mod $[S_n]$. We
prove the latter statement by induction on $\mathrm{dep}(w)$.

If $\mathrm{dep}(w) =1$, that is, $w\in \cm(\Delta_n X)$, let $q :=
P(\star w) \in \calb^\star(\Delta_n X)$. Then $q|_{\phi(u,v)} =
P(\phi(u,v) w) $ and $q|_{\phi(u,v)}$ is normal by $w\in \cm(\Delta_n
X)$. Since
\begin{align*}
\overline{P(\phi(u,v) w)}=\overline{P(\overline{\phi(u,v)} w)}
=\overline{P(P(d(u)P(v)) w)} \leq \overline{P(d(u)P(v)) P(w)} =
\overline{ \overline{\phi(u,v) } P(w)} = \overline{ \phi(u,v) P(w)},
\end{align*}
we have $P(\phi(u,v)w)\equiv 0$ mod $[S_n]$.

Suppose $w\in \cm(\Delta_n X) P( \calb(\Delta_n X))$ and let $w =
w_1P(\tilde{w})$ with $w_1 \in \cm(\Delta_n X)$ and $\tilde{w} \in
\calb(\Delta_n X)$. Since $\mathrm{dep}(\tilde{w})<\mathrm{dep}(w)$,
by the induction hypothesis, we may assume that
$$\phi(u,v) P(\tilde{w})  = \sum_i c_i
p_i|_{s_i},$$ where $0\neq c_i\in \bfk, p_i\in \calb^\star(\Delta_n
X), s_i\in S_n$, $p_i|_{s_i}$ is normal and $\lbar{p_i|_{s_i}} \leq
\lbar{\phi(u,v) P(\tilde{w})}$. Let $q_i := P(w_1p_i)$. Since
$p_i|_{s_i}$ is normal and $w_1\in \cm(\Delta_n X)$, it follows that
$q_i|_{s_i}$ is normal. Furthermore, we have
$$P(\phi(u,v) w) = P( \phi(u,v) w_1 P(\tilde{w})) = \sum_i c_i P(w_1 p_i|_{s_i}) = \sum_i c_i q_i|_{s_i}$$
and $$\lbar{q_i|_{s_i}}= \lbar{P(w_1 p_i|_{s_i})} \leq \lbar{P(w_1
\phi(u,v) P(\tilde{w}) )} = \lbar{P(\phi(u,v) w)} \leq
\lbar{\phi(u,v) P(w)}.$$ Therefore $P(\phi(u,v)w)\equiv 0$ mod
$[S_n]$. This completes the induction. Hence $\phi(u,v) P(w) \equiv 0$ mod $[S_n]$, as needed.
\end{proof}

\begin{lemma}\mlabel{lemma:comtrivial2}
The including compositions in $S_n$ are trivial.
\end{lemma}

\begin{proof}
We need to show that the ambiguities of all possible including
compositions of the polynomials in $S_n$ are trivial. The ambiguities of all such compositions are of the form
$$P(d(u)P(q|_{P(d(v) P(w))})) \text { and } P(d(q|_{P(d(u) P(v))}) P(w)).$$
Let two elements $f$ and $g$ of $S_n$ be
given. They are of the form
$$f := \phi(u,v), \quad g:= \phi(r,s), \quad u,v\in \calb(\Delta_n X)\setminus P(\calb(\Delta_n X)) \text{ and } r,s \in \calb(\Delta_n X).$$

\noindent
{\bf Case I.} Suppose $v=
p|_{\overline{g}}=p|_{\overline{\phi(r,s)}} = p|_{P(d(r)P(s))}$ for
some $p\in \calb^\star (\Delta_n X)$ and $$w := \overline{f} =
\overline{\phi(u,v)}=\lbar{P(d(u)P(v))} = \lbar{
P(d(u)P(p|_{\overline{g}})) } = \lbar{ q|_{\overline{g}}} =
q|_{\overline{g}},$$ with $q= P(d(u)P(p)) \in \calb^\star(\Delta_n
X)$ and $q|_g$ being normal. Then
\begin{align*}
f &=\phi(u,v) = P(d(u)P(p|_{P(d(r)P(s))})) -uP(p|_{P(d(r)P(s))}) +
P(u p|_{P(d(r)P(s))}) + \lambda P(d(u) p|_{P(d(r)P(s))})
\end{align*}
and
\begin{align*}
 q|_g &=
q|_{\phi(r,s)} = P(d(u)P(p|_{P(d(r)P(s))})) - P(d(u)P(p|_{rP(s)})) +
P(d(u)P(p|_{P(rs)})) + \lambda P(d(u)P(p|_{ P(d(r)s)})).
\end{align*}
So we have
\begin{equation}
\begin{aligned}
(f,g)_w :=&f -q|_g \\
=&-uP(p|_{P(d(r)P(s))}) + P(u p|_{P(d(r)P(s))}) + \lambda
P(d(u) p|_{P(d(r)P(s))})\\
& + P(d(u)P(p|_{rP(s)})) -  P(d(u)P(p|_{P(rs)})) - \lambda
P(d(u)P(p|_{ P(d(r)s)})).
\end{aligned}
\mlabel{eq19}
\end{equation}
Since $ \phi(u,v) = P(d(u) P(v))- uP(v)+ P(uv) + \lambda P(d(u) v)$,
we have
{\small
\begin{equation}
\begin{aligned}
-uP(p|_{P(d(r)P(s))}) &=  -uP(p|_{\phi(r,s)}) - uP(p|_{rP(s)}) +
uP(p|_{P(rs)}) + \lambda uP(p|_{P(d(r) s)})\\
P(u p|_{P(d(r)P(s))}) &= + P(up|_{\phi(r,s)}) + P(up|_{rP(s)}) -
P(up|_{P(rs)}) - \lambda P(up|_{P(d(r) s)})\\
\lambda P(d(u) p|_{P(d(r)P(s))}) & = + \lambda P(d(u)
p|_{\phi(r,s)}) + \lambda P(d(u) p|_{rP(s)}) - \lambda P(d(u)
p|_{P(rs)}) - \lambda^2 P(d(u) p|_{P(d(r) s)})\\
P(d(u)P(p|_{rP(s)})) &= \phi(u, p|_{rP(s)}) + uP(p|_{rP(s)}) - P(u
p|_{rP(s)}) -\lambda P(d(u) p|_{rP(s)})\\
-  P(d(u)P(p|_{P(rs)}))&= - \phi(u, p|_{P(rs)}) - uP(p|_{P(rs)}) +
P(u p|_{P(rs)}) + \lambda P(d(u) p|_{P(rs)})\\
- \lambda P(d(u)P(p|_{ P(d(r)s)})) &= - \lambda \phi(u,
p|_{P(d(r)s)}) - \lambda uP(p|_{P(d(r)s)}) + \lambda P(u
p|_{P(d(r)s)}) + \lambda^2 P(d(u) p|_{P(d(r)s)}). \end{aligned}
\label{eq20}
\end{equation}
}
From Eq.~(\ref{eq19}) and Eq.~(\ref{eq20}), it follows that
$$(f,g)_w  = -uP(p|_{\phi(r,s)}) + P(up|_{\phi(r,s)}) + \lambda
P(d(u)p|_{\phi(r,s)}) + \phi(u, p|_{rP(s)})- \phi(u, p|_{P(rs)}) -
\lambda \phi(u, p|_{P(d(r)s)}).$$
By Lemma \mref{operator ideal}, we
have
$$uP(p|_{\phi(r,s)}), P(up|_{\phi(r,s)}), \lambda
P(d(u)p|_{\phi(r,s)})\in Id(S_n) $$ and
$$\phi(u, p|_{rP(s)}), \phi(u,
p|_{P(rs)}), \phi(u, p|_{P(d(r)s)}) \in S_n\subseteq Id(S_n).$$
Since
\begin{align*}
\overline{ uP(p|_{\phi(r,s)}) }, \quad \overline{ P(up|_{\phi(r,s)})
}, \quad \overline{ P(d(u)p|_{\phi(r,s)}) } < \overline{\phi(u,
p|_{\phi(r,s)})}=\overline{\phi(u,v)} =w
\end{align*}
and
\begin{align*}
\overline{\phi(u, p|_{rP(s)})}, \quad \overline{\phi(u,
p|_{P(rs)})}, \quad \overline{\phi(u, p|_{P(d(r)s)})} <
\overline{\phi(u, p|_{\overline{\phi(r,s)}})} =
\overline{\phi(u,v)}=w,
\end{align*}
we have that $(f,g)_w \equiv 0$ mod $(S_n,w)$.
\smallskip

\noindent
{\bf Case II.} Suppose $u =p|_{\overline{g}}
=p|_{\overline{\phi(r,s)}} = p|_{P(d(r)P(s))}$ for some $p\in
\calb^\star(\Delta_n X)$ and $$w := \overline{f} =
\overline{\phi(u,v)}= \lbar{ P(d(u)P(v)) }= \lbar{
P(d(p|_{\overline{\phi(r,s)}}) P(v)) }= \lbar{q|_{\lbar{g}}}=
q|_{\lbar{g}},
$$ with $q= P(d(p)P(v)) \in \calb^\star(\Delta_n X)$ and $q|_g$ being normal.
Then
\begin{align*}
f &=\phi(u,v) = P(d(p|_{P(d(r)P(s))})P(v)) - p|_{P(d(r)P(s))}P(v) +
P(p|_{P(d(r)P(s))} v) + \lambda P(d(p|_{P(d(r)P(s))}) v)
\end{align*}
and
\begin{align*}
 q|_g &= q|_{\phi(r,s)} = P(d(p|_{P(d(r)P(s))})P(v)) -
P(d(p|_{rP(s)})P(v)) + P(d(p|_{P(rs)})P(v)) +\lambda P(d(p|_{P(d(r)
s)})P(v)).
\end{align*}
We have
\begin{align*}
(f,g)_w :=& f -q|_g \\
=& -p|_{P(d(r)P(s))}P(v) + P(p|_{P(d(r)P(s))} v) +\lambda
P(d(p|_{P(d(r)P(s))}) v) \\
& + P(d(p|_{rP(s)})P(v)) - P(d(p|_{P(rs)})P(v)) - \lambda
P(d(p|_{P(d(r) s)})P(v))\\
=& -p|_{\phi(r,s)} P(v)- p|_{rP(s)} P(v)+ p|_{P(rs)} P(v)+ \lambda
p|_{P(d(r)s)} P(v)\\
& +P(p|_{\phi(r,s)} v) + P(p|_{rP(s)} v) - P(p|_{P(rs)} v) -\lambda
P(p|_{P(d(r)s)} v)\\
& +\lambda P(d(p|_{\phi(r,s)}) v) + \lambda P(d(p|_{rP(s)}) v) -
\lambda P(d(p|_{P(rs)}) v) - \lambda^2 P(d(p|_{P(d(r)s)}) v)\\
&+\phi(p|_{rP(s)}, v) + p|_{rP(s)}P(v) - P(p|_{rP(s)} v) - \lambda
P(d(p|_{rP(s)}) v)\\
&- \phi(p|_{P(rs)}, v) - p|_{P(rs)}P(v) + P(p|_{P(rs)} v) + \lambda
P(d(p|_{P(rs)}) v)\\
&- \lambda \phi(p|_{P(d(r) s)}, v) - \lambda p|_{P(d(r) s)}P(w) +
\lambda P(p|_{P(d(r) s)} v) + \lambda^2 P(d(p|_{P(d(r) s)}) v)\\
=&-p|_{\phi(r,s)} P(v) +P(p|_{\phi(r,s)} v)+\lambda
P(d(p|_{\phi(r,s)}) v) + \phi(p|_{rP(s)}, v)- \phi(p|_{P(rs)}, v)-
\lambda \phi(p|_{P(d(r) s)}, v).
\end{align*}
By Lemma \mref{operator ideal}, we have
$$p|_{\phi(r,s)} P(v), P(p|_{\phi(r,s)} v),
P(d(p|_{\phi(r,s)}) v) \in Id(S_n)$$ and
$$ \phi(p|_{rP(s)}, v), \phi(p|_{P(rs)}, v),
\phi(p|_{P(d(r) s)}, v) \in S_n \subseteq Id(S_n).$$ Since
\begin{align*}
\overline{p|_{\phi(r,s)} P(v)}, \quad \overline{P(p|_{\phi(r,s)}
v)}, \quad \overline{P(d(p|_{\phi(r,s)}) v)} <
\overline{\phi(p|_{\phi(r,s)},v)} =\overline{\phi(u,v)}= w
\end{align*}
and
\begin{align*}
\overline{\phi(p|_{rP(s)}, v)}, \quad \overline{\phi(p|_{P(rs)},
v)}, \quad \overline{\phi(p|_{P(d(r) s)}, v)} <
\overline{\phi(p|_{\overline{\phi(r,s)}},v)} =\overline{\phi(u,v)}=
w,
\end{align*}
it follows that $(f,g)_w \equiv 0$ mod $(S_n,w)$.
\end{proof}

By the remark before Lemma~\mref{lemma:comtrivial1}, Lemmas \mref{lemma:comtrivial1} and \mref{lemma:comtrivial2}, it
follows immediately that

\begin{theorem}\label{GSbase}
$S_n$ is a Gr\"{o}bner-Shirshov basis in $\bfk \calb(\Delta_n X)$. Hence $\mathrm{Irr}(S_n)$ in Theorem~\mref{thm:CD lemma} is a linear basis of $\sha(\bfk[\Delta_n X])/\Id(S_n)$.
 \mlabel{lemma:GSbase}
\end{theorem}

\subsection{Bases for free commutative integro-differential algebras}
\mlabel{ss:bases}
We next identify $\mathrm{Irr}(S_n)$ and thus obtaining a canonical basis of $\sha(\bfk[\Delta_n X])/\Id(S_n)$.

\begin{lemma}
Let $\leq$ be the linear order on $C(\Delta X)$ defined in
Eqs.~(\mref{eq:difford}) and (\mref{eq:lex}), and $u = u_0 u_1\cdots u_k \in C(\Delta X)$ with
$u_0,\cdots, u_k \in \Delta X$ and $u_0 \geq \cdots \geq u_k$. Then
$\lbar{d_X(u)} = u_0 u_1\cdots u_{k-1}d_X(u_k)$.
If $u$ is in $C(\Delta_n X)$, then $\lbar{d_X(u)} = u_0 u_1\cdots u_{k-1}d_X(u_k)$ provided $u_k\in \Delta_{n-1} X$.
\mlabel{lemma:diffleadterm}
\end{lemma}

\begin{proof}
We prove the first statement by induction on $k\geq 0$. If $k=0$, then $u=u_0\in \Delta X$ and there is nothing to prove.

Assume the result holds for $k\leq m$, where $m\geq 0$, and consider the case when $k =
m+1$. Then $u = u_0 u_1\cdots u_{m+1}$ with $u_0,\cdots, u_{m+1} \in
\Delta X$ and $u_0 \geq \cdots \geq u_{m+1}$. Let $\hat{u} = u_0
u_1\cdots u_{m}$. Then $$d_X(u) = d_X(\hat{u}u_{m+1}) = \hat{u}
d_X(u_{m+1}) + d_X(\hat{u}) u_{m+1} + \lambda d_X(\hat{u})
d(u_{m+1}).$$ By the induction hypothesis, we have $\lbar{d_X(\hat{u})}
= u_0u_1\cdots d_X(u_m)$. So $\lbar{d_X(\hat{u}) u_{m+1}} =
u_0u_1\cdots d_X(u_m) u_{m+1}$ and $\lbar{d_X(\hat{u}) d_X(u_{m+1}})
= u_0u_1\cdots d_X(u_m) d_X(u_{m+1})$. If $d_X(u_m) \geq u_{m+1}$,
then since $u_m
> d_X(u_m)$ and $u_{m+1} > d_X(u_{m+1})$, we have $$\hat{u}
d_X(u_{m+1})= u_0 u_1\cdots u_{m} d_X(u_{m+1}) > u_0u_1\cdots
d_X(u_m) u_{m+1} > u_0u_1\cdots d_X(u_m) d_X(u_{m+1})$$ and so
$\lbar{d_X(u)} = u_0 u_1\cdots u_{m} d_X(u_{m+1})$. If $u_{m+1}
> d_X(u_m)$ and $u_m >u_{m+1}$, then since $u_m
> d_X(u_m)$, we have $$\hat{u}
d_X(u_{m+1})= u_0 u_1\cdots u_{m} d_X(u_{m+1}) > u_0u_1\cdots
 u_{m+1} d_X(u_m), u_0u_1\cdots d_X(u_m) d_X(u_{m+1})$$ and hence
$\lbar{d_X(u)} = u_0 u_1\cdots u_{m} d_X(u_{m+1})$. If $u_{m+1}
> d_X(u_m)$ and $u_m =u_{m+1}$, then since $u_m
> d_X(u_m)$, we have $$\hat{u}
d_X(u_{m+1})= u_0 u_1\cdots u_{m} d_X(u_{m+1}) = u_0u_1\cdots
 u_{m+1} d_X(u_m)> u_0u_1\cdots d_X(u_m) d_X(u_{m+1})$$ and so
$\lbar{d_X(u)} = u_0 u_1\cdots u_{m} d_X(u_{m+1})$. This completes the induction. The proof of the second statement then follows since under the condition $u_k\in \Delta_{n-1} X$, $d_X(u_k)$ does not change in $\Delta X$ or in $\Delta_n X$.
\end{proof}
We now give the key concept to define $\mathrm{Irr}(S_n)$.
\begin{defn}
Let $u\in \cm(\Delta X)$ with standard form in Eq.~(\mref{eq30}):
\begin{align}
u=u_0^{j_0}\cdots u_k^{j_k}, \text{ where } u_0,\cdots, u_k\in
\Delta X, u_0 > \cdots > u_k \text{ and } j_0,\cdots,j_k\in
\mathbb{Z}_{\geq 1}. \notag
\end{align}
Call $u$ {\bf functional} if either $u\in \{1\} \cup X$ or $j_k
>1$. Denote
$$\cala_f:=\{u\in C(\Delta_n X)\,|\, u \text{ is functional }\}, \quad  \bfk\{X\}_f := \bfk \cala_f \text{ and } A_{f,0} = \bfk (\cala_f\backslash \{1\}). $$
\end{defn}

\begin{prop}
Let $X$ be a finite well-ordered set. Let $(A, d_X):=(\bfk\{ X \}, d_X) := (\bfk[\Delta X], d_X)$ be the
free commutative differential algebra on $X $. Then $A = A_f \oplus d_X(A)$. \mlabel{functional}
\end{prop}

\begin{proof}
We prove the result by induction on $|X|\geq 1$. The case when $|X|=1$ has been proved in~\mcite{GRR}.  Suppose the result holds
for all $X$ such that $|X| < m$ and consider the case when $|X|=m$. Let
$X=\{x_1, x_2, \cdots, x_m\}$ with $x_1 > \cdots > x_m$,  $B = \bfk
\{x_1, \cdots, x_{m-1} \}$ and $C = \bfk\{x_m\}$.
Also denote
$$A_f:=\bfk\{X\}_f,\  B_f:=\bfk\{x_1,\cdots,x_{m-1}\}_f,\ C_f:=\bfk\{x_m\}_f,\ C_{f,0}=\bfk\{x_m\}_{f,0}.$$
By the induction hypothesis, we have
$$B= B_f \oplus d_X(B) \text{ and } C= C_f \oplus d_X(C).$$
Then by the definition of $A_f$, we have
\begin{equation}
A_f =(B_f\ot \bfk) \oplus (B \otimes C_{f,0}) = (B_f\ot
\bfk) \oplus (B_f \ot C_{f,0}) \oplus (d_X(B) \ot C_{f,0}) = (B_f \ot C_f) \oplus (d_X(B) \ot C_{f,0}).
\mlabel{eq:af}
\end{equation}
Therefore $B_f=B_f\ot 1 \subseteq A_f$ and $C_f=1\ot C_f\subseteq A_f$.
Thus $B = B_f \oplus d_X(B) \subseteq A_f + d_X(A)$ and $C = C_f \oplus d_X(C) \subseteq A_f + d_X(A)$. Since $A =B\otimes C $ is generated as an algebra by $B \ot 1$ and $1\ot C$, we have $A \subseteq A_f + d_X(A)$ and so $A_f + d_X(A) = A$.

We are left to show that $A_f \cap d_X(A)= 0$. Let $\calb:=B\cap \cm(\Delta X)$ (resp. $\calb_f:=B_f\cap \cm(\Delta X)$, resp. $\calc:=C\cap \cm(\Delta X)$, resp. $\calc_f:=C_f\cap \cm(\Delta X)$) be the basis of monomials of $B$ (resp. $B_f$, resp. $C$, resp. $C_f$). Then a nonzero element $w$ of $A=B\ot C$ is a sum
$$w=\sum_{i=1}^k u_i\ot \sum_{j=1}^{n_i} \alpha_{ij} v_{ij} = \sum_{i,j}\alpha_{ij} u_i\ot v_{ij} , $$
where $u_1>\cdots>u_k\in \calb, v_{i1}>\cdots>v_{i n_i}\in \calc, 0\neq k_{ij}\in \bfk, 1\leq j\leq n_i.$
Then we have
\begin{equation}
d_X(w)=d_X\left(\sum_{i,j} k_{ij} u_i\ot v_{ij}\right)=\sum_{i,j} \alpha_{ij}\left (d_X(u_i)\ot v_{ij} +u_i \ot d_X(v_{ij}) + \lambda d_X(u_i)\ot d_X(v_{ij})\right).
\mlabel{eq:dsum}
\end{equation}
We distinguish the following three cases.
\smallskip

\noindent
{\bf Case 1.} If $v_{11}\neq 1$, then the leading term in the sum in Eq.~(\mref{eq:dsum}) is $u_1\ot \overline{d_X(v_{11})}$. Since $C_f \cap d_X(C) = 0$, we have $\overline{d_X(v_{11})} \notin \calc_f$. Then $u_1\ot \overline{d_X(v_{11})}\notin \calb\ot \calc_f$. Since  $ \calb \ot
\calc_f$ is a basis of $B \ot C_f$, we have $u_1 \otimes \overline{d_X(v_{11})} \notin B \otimes
C_f$. Therefore $d_X(\sum_{ij} k_{ij} u_i\ot v_{ij})\notin B\ot C_f$. By Eq.~(\mref{eq:af}) we have
$$B \otimes C_f = B_f \ot C_f \oplus d_X(B) \ot C_f = B_f \ot C_f
\oplus d_X(B) \ot C_{f,0} \oplus d_X(B) \ot \bfk =A_f \oplus d_X(B)
\ot \bfk.$$
Therefore $d_X(w)\notin A_f$.

{\bf Case 2.} If $v_{11}=1$ and either $\lbar{d_X(u_1)}>u_2$ or $\lbar{d_X(u_1)}=u_2$ and $v_{21}=1$, then since $d_X(1)=0$, by the definition of the order defined on $\Delta X$, the leading term in the sum in Eq.~(\mref{eq:dsum}) is $\overline{d_X(u_1)}\ot 1$ where $\overline{d_X(u_1)}\in \frakB$ denotes the leading term of $d_X(u_1)$. Since $B_f\cap d_X(B)=0$, we have $\overline{d_X(u_1)}\ot 1\notin \calb_f$. Then $\overline{d_X(u_1)}\ot 1\notin \calb_f\ot \calc$ and hence not in $B_f\ot C$. Also $1\notin \calc_{f,0}$ implies that $\overline{d_X(u_1)}\ot 1\notin \calb\ot \calc_{f,0}$. Here $\calc_{f,0}=\calc\backslash \{1\}$ is the standard basis of $C_{f,0}$. Thus
$\overline{d_X(u_1)}\ot 1\notin (\calb_f\ot \calc)\cup (\calb\ot \calc_{f,0}).$
Then we have $\overline{d_X(u_1)}\ot 1 \notin (B_f \ot C) + (B\ot C_{f,0})$ and hence $d_X(\sum_{ij} k_{ij} u_i\ot v_j) \notin (B_f \ot C) + (B\ot C_{f,0})$. Then $d_X(w)$ is not in $A_f$ by Eq.~(\mref{eq:af}).
\smallskip

\noindent
{\bf Case 3.} If $v_{11}=1$, $\lbar{d_X(u_1)}= u_2$ and $v_{21}\neq 1$ (note that $\lbar{d_X(u_1)}<u_2$ is impossible since $u_1>u_2$), then the leading term of the sum in Eq.~(\mref{eq:dsum}) is $u_2\ot \lbar{v_{21}}$. Then the proof is the same as for Case 1.
\smallskip

In summary, we have proved that $d_X(w)\notin A_f$ and hence $A_f \cap d_X(A)= 0$.
\end{proof}

\begin{lemma}
Let $A_f=\bfk\{X\}_f$, $A_n = \bfk [\Delta_n X]$, $A_{n,f} = A_{n}\cap A_f$ and $d_{A_n}$ to be the restriction $d|_{A_n}$ except $d(x^{(n)})=0$ for $x\in X$. Then
$A_n = A_{n,f}\oplus d_{A_n}(A_n)$. \mlabel{lemma:decom}
\end{lemma}

\begin{proof}
Since $A_{n,f} \subset A_f$, $d_{A_n}(A_{n}) \subseteq d_X(A)$ and
$A_f \cap d_X(A) = 0 $ by Proposition~\ref{functional}, we have $A_{n,f}
\cap \mathrm{im}(d_{A_n}) =0 $. Thus we only need to show $A_n\subseteq A_{n,f}+d_{A_n}(A_n)$ since $A_n\supseteq A_{n,f}+d_{A_n}(A_n)$ is clear. Suppose $A_n\not \subseteq A_{n,f}+d_{A_n}(A_n)$. There is a monomial $u\in C(\Delta_n X)$ in $A_n\backslash (A_{n,f}+d_{A_n}(A_n))$ that is minimal under the order $\leq_n$ on $\cm(\Delta_n X)$ defined in Eqs.~(\mref{eq:lex}) and (\mref{eq:difford}).
Then $u\not\in A_f$.  Assume the minimum variable in $u$ is $x$ and
$\ell$ is the highest differential order of $x$ in $u$. Then $u$ can be
expressed as $u = \hat{u} (x^{(\ell-1)})^{m} x^{(\ell)}$ with
$\hat{u}\in C(\Delta X)$ and $m\geq 0$. Let $v =
\hat{u}(x^{(\ell-1)})^{m+1} \in C(\Delta X)$. By Lemma
\mref{lemma:diffleadterm}, we have $u = \lbar{ d_{A_n}(v)}$.
So we can write $u =  d_{A_n}(v) - \sum_i c_iu_i$ with $0\neq c_i \in \bfk$ and $u> u_i$. Then $d_{A_n}(v)\in d_{A_n}(A_n)$ and $u_i \in A_{n,f} +
\mathrm{im}(d_{A_n})$ by the minimality of $u$ in $A_n\backslash (A_{n,f} +
\mathrm{im}(d_{A_n}))$. Thus $u\in
A_{n,f}+ \mathrm{im}(d_{A_n})$. This is a contradiction.
\end{proof}

\begin{lemma}\mlabel{decomposition}
\begin{enumerate}
\item
Let $\cala_{d} := \{ \lbar{ d_X(u)} \mid u\in C(\Delta X) \}$ and
$\cala_{f} := \{ u \in C(\Delta X) \mid u \text{ is functional} \}$.
Then $C(\Delta X)$ is the disjoint union of $\cala_{d}$ and $\cala_{f}$,
that is, $C(\Delta X) = \cala_{d} \sqcup \cala_{f}$.
\mlabel{it:deca}
\item
We have $C(\Delta_n X)= (\cala_{d} \cap C(\Delta_n X)) \sqcup (\cala_{f} \cap C(\Delta_n X))$.
\mlabel{it:decb}
\end{enumerate}
\end{lemma}

\begin{proof}
\mref{it:deca} First we show that $\cala_{d} \cap \cala_{f} =\emptyset$. Let $\lbar{
d_X(u)} \in \cala_{d}$ with $u\in C(\Delta X)$. Suppose the standard
expression of $u$ is ${\hat u} (x^{(\ell)})^{m}$ for some ${\hat u}
\in C(\Delta X)$. Thus
$$\lbar{ d_X(u)} = \lbar{ {\hat u} d_X((x^{(\ell)})^{m}) }=  \lbar{ {\hat u}
(x^{(\ell)})^{m-1} x^{(\ell+1)}} =   \lbar{{\hat u}}
(x^{(\ell)})^{m-1} x^{(\ell+1)}$$ and so $\lbar{d_X(u)} \notin
\cala_{f}$. Next we show that $C(\Delta X) = \cala_{d} \cup \cala_{f}$. Let
$u\in C(\Delta X) \setminus  \cala_{f}$. Suppose the minimum variable in
$u$ is $x$ and $\ell$ is the largest differential degree of $x$.
Then $u$ can be expressed as $u = \hat{u} (x^{(\ell-1)})^{m}
x^{(\ell)}$ with $\hat{u}\in C(\Delta X)$ and $m\geq 0$. Let $v =
\hat{u}  (x^{(\ell-1)})^{m+1} \in C(\Delta X)$. By Lemma
\mref{lemma:diffleadterm}, we have that $u = \lbar{ d_X(v)} \in
\cala_{d}$. Hence $C(\Delta X) = \cala_{d}\sqcup \cala_{f}$.
\smallskip

\mref{it:decb} Since $C(\Delta_n X) \subseteq C(\Delta X)$, the result holds from Item~\mref{it:deca}.
\end{proof}

\begin{theorem}
 \mlabel{PBW Base}
Let $A_n, A_{n,f}$ be as defined in Lemma~\mref{lemma:decom} and let $I_{ID,n}$ be the differential Rota-Baxter ideal of
$\sha(A_n)$ generated by $S_n$ in Eq.~(\mref{eq2}). Then as tensor product of modules
$$\sha(A_n) / I_{ID,n} \cong A_n \oplus \left(\bigoplus_{k\geq 0} A_n\otimes A_{n,f}^{\otimes k} \otimes A_n \right).$$
\end{theorem}

\begin{proof}
For any $s =\phi(u,v) \in S_n$, by Lemma \ref{lemma:leading term},
we have $\lbar{s} = 1\otimes \lbar{d(u_1)} \otimes w$, where
$\lbar{d(u_1)} \in \cala_{d} \cap C(\Delta_n X)$ and $w \in \calb(\Delta_n X)$. Recall
that
$$\calb(\Delta_n X) = \sqcup_{m\geq 1} C(\Delta_n X)^{\otimes m} =
\{a_1  \otimes \cdots\otimes a_m \mid a_1,\cdots,a_m \in C(\Delta_n
X),  m\geq 1 \}.$$ By Theorems \ref{CD lemma} and
\ref{GSbase}, and Lemmas \ref{lemma:decom} and \ref{decomposition}, we have
\begin{align*}
\mathrm{Irr(S_n)}&= \calb(\Delta_n X) \setminus \left\{ q|_{\overline{s}}
\mid q\in \calb^\star(\Delta_n X), s\in S_n , q|_{s}\text{ is normal}\right\}\\
&=\calb(\Delta_n X) \setminus \{ q|_{ 1\otimes \lbar{d(u_1)} \otimes
w} \mid q\in \calb^\star(\Delta_n X),\lbar{d(u_1)} \in \cala_{d} \cap C(\Delta_n X), w\in \calb(\Delta_n X) \}\\
&=\calb(\Delta_n X) \setminus \{ a_1 \otimes \cdots \otimes a_k
\in C(\Delta_n X)^{\otimes k} \mid a_i\in \cala_{d} \cap C(\Delta_n X) \text{ for some } 1<i<k, k\geq 1 \}\\
&= \{a_1 \otimes \cdots\otimes a_k \mid a_1, a_k \in
C(\Delta_n X), a_i \in \cala_{f} \cap C(\Delta_n X) \text{ for } 1<i<k, k\geq 1\}
\end{align*}
is a $\bfk$-basis of  $\bfk \calb(\Delta_n X) / I_{ID}$. Since $A_n
= \bfk C(\Delta_n X)$ and $A_{n,f} = \bfk \cala_f\cap C(\Delta_n X)$, the theorem follows.
\end{proof}

Let
\begin{equation}
S  := \left\{P(d(u) P(v))- uP(v)+ P(uv) + \lambda P(d(u) v) \big| u\in
\sha(\Delta_n X) \setminus P( \sha(\Delta_n X)), v\in \sha(\Delta
X) \right\}.
\mlabel{eq:gsid}
\end{equation}

\begin{lemma} \label{idealcompa}
Let $I_{ID,n}$ $($resp. $I_{ID}$$)$ be the differential Rota-Baxter ideal
of $\sha(\Delta_n X)$ $($resp. $\sha(\Delta X)$$)$ generated by
$S_n$ $($resp. $S$$)$. Then as $\bfk$-modules we have $I_{ID,1} \subseteq I_{ID,2} \subseteq \cdots
$, $ I_{ID} = \cup_{n\geq 1} I_{ID,n}$ and $I_{ID,n} = I_{ID} \cap
\bfk \sha(\Delta_n X)$. \mlabel{lemma:idealcompa}
\end{lemma}

\begin{proof}
Since $\bfk \sha(\Delta_n X) \subseteq \bfk \sha(\Delta_{n+1} X)$
for any $n{\geq 1}$, we have $I_{ID,1}
\subseteq I_{ID,2} \subseteq \cdots $ and $ I_{ID} = \cup_{n\geq 1}
I_{ID,n}$ by Eq.~(\ref{eq2}). We next show $I_{ID,n} = I_{ID}
\cap \sha(\Delta_n X)$. Obviously, $I_{ID,n} \subseteq I_{ID}
\cap \sha(\Delta_n X)$. So we only need to verify $I_{ID} \cap \sha(\Delta_n X) \subseteq I_{ID,n}$. By Theorem
\ref{PBW Base} we have

$$\sha(\Delta_n X) \cong
\left( A_n \oplus \left(\bigoplus_{k\geq 0} A_n\otimes A_{n,f}^{\otimes k}
\otimes A_n \right)\right) \oplus I_{ID,n}.$$ Let $$J_n := A_n \oplus
\left(\bigoplus_{k\geq 0} A_n\otimes A_{n,f}^{\otimes k} \otimes A_n \right).$$
Then $\sha(\Delta_n X) = J_n \oplus I_{ID,n}$ and $J_1
\subseteq J_2 \subseteq \cdots $. Let $n,k{\geq
1}$. Since $J_{n+k} \cap I_{ID,n+k} = 0$ and $J_{n} \subseteq J_{n+k}$, we have $J_{n}\cap I_{ID,n+k} = 0$. Since $I_{ID,n}
\subseteq I_{ID,n+k}$, by modular law we have
\begin{equation}I_{ID,n+k}
\cap \sha(\Delta_n X) = I_{ID,n+k} \cap (J_n \oplus I_{ID,n})
= (I_{ID,n+k}\cap J_n) \oplus I_{ID,n} = I_{ID,n}.
\label{eq15}
\end{equation}

Let $u\in I_{ID} \cap \sha(\Delta_n X)$. By $I_{ID}
=\cup_{n\geq 1} I_{ID,n}$, we have $u\in I_{ID,N}$ for some $N \in
\mathbb{Z}_{\geq 1}$. If $N\geq n$, by Eq.~(\ref{eq15}), $u \in
I_{ID,N} \cap \sha(\Delta_n X) = I_{ID,n}$. If $N<n$, then
$u\in I_{ID,N} \subseteq I_{ID,n}$. Hence $I_{ID} \cap \sha(\Delta_n X) \subseteq I_{ID,n}$ and so $I_{ID} \cap\sha(\Delta_n X) = I_{ID,n}$.
\end{proof}

Now we are ready to prove the main result of this paper.

\begin{theorem} \mlabel{thm:gsb}
$($=Theorem~\mref{thm:gsb0}$)$
Let $X$ be a nonempty well-ordered set. Let $\sha(\bfk\{X\})=\sha(\Delta X)$ be the free commutative differential Rota-Baxter algebra on $X$. Let $I_{ID}$ be the ideal of $\sha(\bfk\{X\})$ generated by $S$ defined in Eq.~(\mref{eq:gsid}). Then the composition
$$\sha(A)_f:=A \oplus \left(\bigoplus_{k\geq 0} A\otimes A_{f}^{\otimes k} \otimes A \right) \hookrightarrow  \sha(A) \to \sha(A) / I_{ID}$$
of the inclusion and the quotient map is a linear isomorphism. In other words,
$$\sha(A)=\sha(A)_f\oplus I_{ID}.$$
\end{theorem}

\begin{proof}
First assume that $X$ is finite. Denote $A=\bfk[\Delta X]$ and $A_n=\bfk[\Delta_n X], n\geq 1$. By Theorem \ref{PBW Base} and Lemma \ref{idealcompa} we have the linear isomorphisms
$$A_n \oplus \left(\bigoplus_{k\geq 0} A_n\otimes A_{n,f}^{\otimes k} \otimes A_n \right) \cong  \sha(\Delta_n X) / I_{ID,n} =  \sha(\Delta_n X) / (I_{ID}
\cap \sha(\Delta_n X) ) \cong (\sha(\Delta_n X) +
I_{ID}) / I_{ID}$$
that are compatible with the direct system $\dirlim A_n$. Since $A=\dirlim A_n$ as $\bfk$-module, we have
$$ A \oplus \left(\bigoplus_{k\geq 0} A\otimes A_{f}^{\otimes k} \otimes A \right) =\dirlim \left(A_n \oplus \left(\bigoplus_{k\geq 0} A_n\otimes A_{n,f}^{\otimes k} \otimes A_n \right)\right)
\cong \dirlim ((\sha(\Delta_n X) + I_{ID})
/ I_{ID}) = \sha(A) / I_{ID}.$$

Now let $X$ be a nonempty well-ordered set. Let $Y$ be a finite subset of $X$. Denote $A_{X,f}=A_f, A_{Y,f}=\bfk\{Y\}_f$.
Then by the definition of $A_f$ we have
\begin{equation}
A_Y\cap A_{X,f}=A_{Y,f} \ \text{ and }\ d_X(A_Y)=d_Y(A_Y).
\mlabel{eq:fy}
\end{equation}

Let $a\in A_X$. Then there is a finite $Y\subseteq X$ such that $a\in A_Y$. Thus by Proposition~\mref{functional}, we have $a\in A_{Y,f}+d_X(A_Y)$ which is contained in $A_{X,f}+d_X(A_X)$ by Eq.~(\mref{eq:fy}). Thus $A_X=A_{X,f}+d_X(A_X)$.
On the other hand, let $0\neq a\in d_X(A_X)$. Then $a=d_X(b)$ for $b\in A_X$. Then there is a finite $Y\subseteq X$ such that $b\in A_Y$ and hence $a\in d_Y(A_Y)$. Then by Proposition~\mref{functional} and Eq.~(\mref{eq:fy}), we have $a\notin A_{Y,f}=A_Y\cap A_{X,f}$. Hence $a\notin A_{X,f}$. This proves $A_{X,f}\cap d_X(A_X)=0$. Hence $A_X=A_{X,f}\oplus d_X(A_X)$.

Now let $u\in \sha(A_X)$. Then there is a finite subset $Y\subseteq X$ such that $a\in \sha(A_Y)$. Then by the case of finite sets proved above, $u\in \sha(A_Y)_f + I_{Y,ID}$. By definition, $\sha(A_Y)_f\subseteq \sha(A)_f$ and $I_{Y,ID} \subseteq I_{ID}$. Hence $u\in \sha(A)_f + I_{ID}$. Further, if $0\neq u\in I_{ID}$, then there is a finite $Y\subseteq X$ such that $u\in I_{Y,ID}$. Thus $u$ is not in $\sha(A_Y)_f$ since $\sha(A_Y)_f\cap I_{Y,ID}=0$.
By the definition of $\sha(A_X)_f$, we have $\sha(A_Y)\cap \sha(A_X)_f = \sha(A_Y)_f$. Therefore $u$ is not in $\sha(A_X)_f$. This proves $\sha(A_X)=\sha(A_X)_f \oplus I_{X,ID}$\,.
\end{proof}

\smallskip

\noindent
{\bf Acknowledgements}:
Xing Gao thanks support from NSFC grant 11201201. Li Guo acknowledges support from NSF grant DMS~1001855.

\end{document}